%% file: main.tex
\documentclass[msom,final,nonblindrev]{informs3} % Use this for listing out authors' information submissions

\OneAndAHalfSpacedXI % % default line spacing !!!!!!!!!!!!!!!!!!!!!!!
\usepackage{graphicx}
\usepackage{subfigure, epsfig}
\usepackage{natbib}
\usepackage{algorithm,algorithmicx,algpseudocode}
\usepackage{booktabs, multirow}

 \bibpunct[, ]{(}{)}{,}{a}{}{,}%
 \def\bibfont{\small}%
\TheoremsNumberedThrough     % Preferred (Theorem 1, Lemma 1, Theorem 2)
\ECRepeatTheorems

\EquationsNumberedThrough    % Default: (1), (2), ...
\newcommand{\E}[3]{\mathbb{E}_{#1\sim f_{#1}(#2,#1)}[#3(#2,#1)]}
\newcommand{\nablaE}[3]{\nabla_{x}\mathbb{E}_{#1\sim f_{#1}(#2,#1)}[#3(#2,#1)]}
\newcommand{\Enabla}[3]{\mathbb{E}_{#1\sim f_{#1}(#2,#1)}[\nabla_{x}#3(#2,#1)]}

\begin{document}
%%%%%%%%%%%%%%%%

% Outcomment only when entries are known. Otherwise leave as is and
%   default values will be used.
%\setcounter{page}{1}
%\VOLUME{00}%
%\NO{0}%
%\MONTH{Xxxxx}% (month or a similar seasonal id)
%\YEAR{0000}% e.g., 2005
%\FIRSTPAGE{000}%
%\LASTPAGE{000}%
%\SHORTYEAR{00}% shortened year (two-digit)
%\ISSUE{0000} %
%\LONGFIRSTPAGE{0001} %
%\DOI{10.1287/xxxx.0000.0000}%

% Author's names for the running heads
% Sample depending on the number of authors;
% \RUNAUTHOR{Jones}
% \RUNAUTHOR{Jones and Wilson}
% \RUNAUTHOR{Jones, Miller, and Wilson}
% \RUNAUTHOR{Jones et al.} % for four or more authors
% Enter authors following the given pattern:
\RUNAUTHOR{Wenxuan Liu and Zhihai Zhang $^*$}

% Enter the (shortened) title:
\RUNTITLE{Solving Data-Driven Newsvendor Pricing Problems with Decision-Dependent Effect}

% Full title. Sample:
\TITLE{Solving Data-Driven Newsvendor Pricing Problems with Decision-Dependent Effect}

% Block of authors and their affiliations starts here:
% NOTE: Authors with same affiliation, if the order of authors allows,
%   should be entered in ONE field, separated by a comma.
%   \EMAIL field can be repeated if more than one author
\ARTICLEAUTHORS{%
    \AUTHOR{Wenxuan Liu}
	\AFF{Department of Industrial Engineering, Tsinghua University, Beijing, 100084, China}
	\AUTHOR{Zhi-Hai Zhang$^*$}
	\AFF{Department of Industrial Engineering, Tsinghua University, Beijing, 100084, China, zhzhang@tsinghua.edu.cn}
% Enter all authors
} % end of the block

\ABSTRACT{%
This paper investigates the data-driven pricing newsvendor problem, which focuses on maximizing expected profit by deciding on inventory and pricing levels based on historical demand and feature data. We first build an approximate model by assigning weights to historical samples. However, due to decision-dependent effects, the resulting approximate model is complicated and unable to solve directly. To address this issue, we introduce the concept of approximate gradients and design an Approximate Gradient Descent (AGD) algorithm. We analyze the convergence of the proposed algorithm in both convex and non-convex settings, which correspond to the newsvendor pricing model and its variants respectively. Finally, we perform numerical experiment on both simulated and real-world dataset to demonstrate the efficiency and effectiveness of the AGD algorithm. We find that the AGD algorithm can converge to the local maximum provided that the approximation is effective. We also illustrate the significance of two characteristics: distribution-free and decision-dependent of our model. Consideration of the decision-dependent effect is necessary for approximation, and the distribution-free model is preferred when there is little information on the demand distribution and how demand reacts to the pricing decision. Moreover, the proposed model and algorithm are not limited to the newsvendor problem, but can also be used for a wide range of decision-dependent problems.
}%

%Use this for final submission
%\HISTORY{This paper was first submitted on January 1, 2021 and has been with the authors for 3 months for 2 revisions.}

\maketitle
%%%%%%%%%%%%%%%%%%%%%%%%%%%%%%%%%%%%%%%%%%%%%%%%%%%%%%%%%%%%%%%%%%%%%%

% Samples of sectioning (and labeling) in POMS
% NOTE: (1) \section and \subsection do NOT end with a period
%       (2) \subsubsection and lower need end punctuation
%       (3) capitalization is as shown (title style).
%
%\section{Introduction.}\label{intro} %%1.
%\subsection{Duality and the Classical EOQ Problem.}\label{class-EOQ} %% 1.1.
%\subsection{Outline.}\label{outline1} %% 1.2.
%\subsubsection{Cyclic Schedules for the General Deterministic SMDP.}
%  \label{cyclic-schedules} %% 1.2.1
%\section{Problem Description.}\label{problemdescription} %% 2.

% Text of your paper here

\input{introduction.tex}

\input{model.tex}

\input{analysis.tex}

\input{experiment.tex}

\input{conclusion.tex}

\ACKNOWLEDGMENT{Wenxuan Liu acknowledges the support from the department of Industrial Engineering, Tsinghua University and is grateful for the support from the mentor. The authors are indebted to the department editor, senior editor, and referees for their valuable and constructive suggestions.}

%%REFERENCES%%
%%%%%%%%%%%%%%%%%%%%%%%%%%%%%%%%%%%%%%%%%%%%%%%%%%%%%%%%%%%%%%%%%%%%%%%%%%%%%%%%%%%%%%%%%%%%%%%%%%%%%%%%%%%%%%%%%%%%%%%%%%%%%%%%%%%%
%% This template complies references using bibtex. You will need to use pomsref.bst file for biblography style.
%REFERENCES USING BIBTEX FILES
%%%%%%%%%%%%%%%%%%%%%%%%%%%%%%%%%%%%%%%%%%%%%%%%%%%%%%%%%%%%%%%%%%%%%%%%%%%%%%%%%%%%%%%%%%%%%%%%%%%%%%%%%%%%%%%%%%%%%%%%%%%%%%%%%%%%

\bibliographystyle{pomsref}

 \let\oldbibliography\thebibliography
 \renewcommand{\thebibliography}[1]{%
    \oldbibliography{#1}%
    \baselineskip14pt %Change this for line spacing within the same reference
    \setlength{\itemsep}{10pt}% %Change this for spacing between two referneces
 }
\bibliography{reference.bib}
%%%%%%%%%%%%%%%%%%%%%%%%%%%%%%%%%%%%%%%%%%%%%%%%%%%%%%%%%%%%%%%%%%%%%%%%%%%%%%%%%%%%%%%%%%%%%%%%%%%%%%%%%%%%%%%%%%%%%%%%%%%%%%%%%%%%

%Hayes, R. H., G. P. Pisano. 1996. Manufacturing strategy: At the intersection of two paradigm shifts. Production and Operations Management, 5 (1), 25-41.

%%%%%%%%%%%%%%%%%%%%%%%%%%%%%%%%%%%%%%%%%%%%%%%%%%%%%%%%%%%%%%%%%%%%%%%%%%%%%%%%%%%%%%%%%%%%%%%%%%%%%%%%%%%%%%%%%%%%%%%%%%%%%%%%%%%%
%% %If you don't use BiBTex, you can manually itemize references as shown in the referneces for the electronic comapanion. See below.
 %%%%%%%%%%%%%%%%%%%%%%%%%%%%%%%%%%%%%%%%%%%%%%%%%%%%%%%%%%%%%%%%%%%%%%%%%%%%%%%%%%%%%%%%%%%%%%%%%%%%%%%%%%%%%%%%%%%%%%%%%%%%%%%%%%%%

%% Here starts the e-companion (EC). Place your appendix content here.
%%%%%%%%%%%%%%%%%%%%%%%%%%%%%%%%%%%%%%%%%%%%%%%%%%%%%%%%%%
\ECSwitch % Comment this line out if you do not need e-companion.
%%%%%%%%%%%%%%%%%%%%%%%%%%%%%%%%%%%%%%%%%%%%%%%%%%%%%%%%%%

%%% Main head for the e-companion
\ECHead{E-Companion for Solving Data-Driven Newsvendor Pricing Problems with Decision-Dependent Effect}

% \noindent A general heading for the whole e-companion should be provided here as in the example above this paragraph.

\input{appendix.tex}

%If you don't use BiBTex, you can manually itemize references as shown below.

% \begin{thebibliography}{}

%     \bibitem[{Doe et~al.(2006)}]{Doe2014}
%     Doe, F., J. Smith, G. Toledano, R. Richard. 2006.
%     Title of the article here. {\it Journal of Science,} 89 (5), 882--887.

%         \bibitem[{Smith et~al.(2016)}]{Doe2009}
% Smith, A., J. Roe, R. Doe.  2016.
%     Another article title. {\it Another Journal Name,} 9 (1), 85--87.

% \end{thebibliography}

% Use this for adding Appendix
% Options are (1) APPENDIX (with or without general title) or
%             (2) APPENDICES (if it has more than one unrelated sections)
% Outcomment the appropriate case if necessary
%
% \begin{APPENDIX}{<Title of the Appendix>}
% \end{APPENDIX}
%
%   or
%
% \begin{APPENDICES}
% \section{<Title of Section A>}
% \section{<Title of Section B>}
% etc
% \end{APPENDICES}

%%%%%%%%%%%%%%%%%
\end{document}

%% file: introduction.tex
\section{Introduction}

The newsvendor pricing problem plays an important role in the inventory decision theory. Nowadays, the newsvendor pricing model is applied to various fields, such as retailing, energy industry, service, and agriculture \citep{Shen-2022-POM,ijoo-price-setting-2021,price-only-Deyong2019}. Compared with the traditional newsvendor model, the newsvendor pricing model has the decision-dependent property. In traditional newsvendor problems, the order decision is independent of model parameters. But in the newsvendor pricing model,  the decision can affect the distribution of random parameters such as demand. We call this effect as decision-dependent effect. For example, when a charging station is to determine the electrical supply and charging price, the customers' charging demand is influenced by its price decision. The decision-dependent effect can increase the complexity of model since the manager need to predict how the distributions change with decisions rather than predict a static distribution.  % version 1

Recently, the data-driven newsvendor is also attracting attention since the demand distribution is often unknown in practice \citep{OR-big-data-newsvendor,ijoo-price-setting-2021}. The decision maker is only accessible to the historical data. Those data often contains historical decision, random parameter and features. For instance, consider the scenario again that a charging station should determine the electricity supply and charging price in a daily frequency. First, the manager does not know the distribution of demand under specific price, but he/she is provided with the historical demand data and some feature information including date, whether and time. Those feature can also affect the distribution of random parameter. Such historical information enables the decision maker to forecast or prescribe the demand by data-driven approaches. % version 1

% 写cost of price adjustment和risk-averse
% The widely accepted criterion for the newsvendor pricing problem is to maximize the expected profit. However, other factors such as the risk and cost of price adjustment also play essential role in the pricing decision. The cost of price adjustment is one of the factors that the decision maker should consider. For instance, the change of price can lead to the time and attention cost for managers "to gather the relevant information and implement decisions" \citep{adjustment-cost1-BERGEN2003}. The menu costs unfair feelings of customers are also not costless \citep{adjustment-cost2-Chen2011}. For these reasons, we extend the traditional newsvendor pricing problem and further consider the price-adjustment costs.

% 讲解决这类问题的困难: decision affect uncertainty, distribution free, constraints, profit loss (non linear).

% motivation: 1. 我们的motivation来自于实际生产场景中改变产品的定价或者给新产品给予定价的情况。在data-driven的前提下，我们往往只有该产品在其他需求下的定价，或者相似产品的定价和需求信息，而管理者需要通过这些信息同时预测需求随价格的变化规律以及需求波动的分布。在多数情况下，管理者可以通过回归的方法... 但是这样要么误差很大，要么没法优化。因此需要一套将预测与优化整合起来的方法

Our motivation stems from real-world production scenarios like assigning a new price to products. Given the inherent uncertainty of demand and the lack of information on its relationship to price, the pricing decision is challenging. This is particularly the case for the pricing of new products, where the manager only have the historical price and demand data of similar products. In this case, the manager is exactly a newsvendor and need to decide how many products to produce and how to set price for them. Another application is in the operation management of the electricity industry, where managers need to combine the historical data and today's weather features to determine the electricity price. Because of the decision-dependent effect, the demand model consists of the demand-price relationship and the demand fluctuation. The manager can adopt the Predict-then-Optimize (PTO) framework that first predicts the demand rule by regression and then optimizes by predicted demand rule. However, if the true demand model is complex, the estimation error will be large and the subsequent optimization step will further amplify this error \citep{OR-big-data-newsvendor}. Therefore, a model that integrates the prediction and optimization steps is needed. Ideally, we hope that this model will be applicable under various complex demand models and easy to solve.

The difficulty of our work lies in many aspects. The main difficulty that set our work apart from other traditional newsvendor model is the decision-dependent distribution of demand. We cannot determine the demand distribution in advance since it is affected by decisions. For this reason, we integrate the prediction and optimization steps by the prescriptive approach. But the decision variable should be assigned into the sample weight and bring nonlinearity and even non-continuity to the problem. The second barrier is the distribution-free setting. We can only use historical data to approximate the conditional demand distribution without making any assumption about the concrete form of the distribution.

% 介绍自己的工作

In this work, we study a data-driven newsvendor pricing problem with decision-dependent property. We investigate how prescriptive methods are used to approximate the objective functions and validate the asymptotical optimality of the prescriptive approach. Because the decision-dependent approximate functions are highly nonlinear and difficult to work with, significant effort revolves around deriving an approximate gradient for the objective function. To show the validity of the approximate gradient, we analyze the convergency of the approximate gradient. Based on our approximate gradient, we develop an approximate gradient descent (AGD) algorithm in which the descent direction is determined by the approximate gradient. To validate the proposed AGD algorithm, we prove that any converge subsequences generated by AGD can converge to the points that satisfies the necessary condition for optimality. We also study the extended price-only model and price-adjustment cost model, which satisfies the convex and strongly convex conditions respectively. We show that under convex condition and certain assumptions, the objective error of AGD algorithm is bounded, indicating the theoretical guarantee when applying to our newsvendor pricing model. We further prove that the sequence of solutions provided by AGD converges to the true optimal point under the strongly convex conditions, which is in accordance with the price-adjustment cost scenario. Finally, we apply our method and algorithm to both simulated and practical datasets and confirm the effectiveness of our approach. We also demonstrate the importance of the two key properties of our model, distribution-free and decision-dependent.

% contribution: 1. 为DAU的定价问题(带约束非线性)提出prescriptive model, 证明其渐进最优和渐进可行性；但发现近似模型非常难处理 2. 提出approximate gradient, 如何derive, 收敛到stable point, which allows us to... 3. 证明了基于approximate gradient的算法的收敛性(强凸，一般凸)，分别对应带与不带价格变动成本  4. 数值实验

\subsection{Model and Approach Overview}

We consider a single product scenario and the decision variables are the order quantity $q\in \mathbb{R}^+$ and price $p\in \mathbb{R}^+$. The decision maker aims to minimize its expected cost. He is provided with the historical selling information $z$ that is related to the demand $d$. Note that both the feature and the price decision can affect the distribution of demand $D\sim f_D(p,z)$. The company tries to minimize the expected cost, $\mathbb{E}_{D\sim f_D(p,z)}[l(p,q,D)]$. The true model can be represented below: 

% price-adjustment cost 在pi中体现
\begin{align}
	\min_{p,q}\mathbb{E}_{D\sim f_D(p,z)}[l(p,q,D)], \label{eq::true-obj-pi}
\end{align}

Note that the distribution of demand $f_D(p,z)$ is unknown. Although previous studies have proposed several widely-used demand models, for example the additive demand \citep{additive-demand-Biswas2018,additive-demand-Chong2015} and multiplicative demand \citep{two-side-11-tech-note-or,multiplicative-demand-Salinger2011} we do not make any assumption on the concrete form of $f_D(p,z)$. We use the feature information to link demand of the following period with historical demand data as in \cite{Bertsimas-2019}'s work. Based on historical data, some local machine learning (ML) methods, including $k$-nearest neighbors (kNN) and kernel methods are adopted to construct the weights from data. The weights reflect the similarity between a historical sample and the present selling scenario. We construct an approximate model with such weights:

\begin{equation*}
	\min_{p,q}\mathbb{\hat E}_{D\sim f_D(p,z)}[l(p,q,D)]
\end{equation*}

where $\mathbb{\hat E}_{D\sim f_D(p,z)}[l(p,q,D)]$ approximates the objective function. The model is still difficult to solve because the approximate functions are nonconvex and even discontinuous if the kNN or CART weight is adopted. Therefore, we develop the approximate gradient as follows: 

\begin{equation*}
	G^N(x;z) \overset{def}{=} \mathbb{\hat E}_{D\sim f_D(p,z)}[\nabla_{p,q} l(p,q,D)]
\end{equation*}

where $G^N(x;z)$ approximates the expectation of the objective gradient. The approximate gradient can then be used to solve the true model.

Our contributions may be summarized as follows: 

1. We formally develop a data-driven approach to prescribe the distribution-free newsvendor pricing problem and extend to the price-only and price-adjustment cost scenarios. Our approach approximates the objective by weight functions generated by machine learning algorithms (nearest neighbor, kernel, tree methods). We prove the convergence performance for our approximate model. Unfortunately, the approximate model can be non-differentiable and discontinuous because of the decision-dependent effect, implying that solving the approximate model may be challenging.

2. Given the intractability of the approximate model, we develop an approximation to the gradient of true objective function, which we call the approximate gradient. The approximate gradient is derived by the same weight approximation approach as the objective functions. We prove a key consistency result of the approximate gradient function in Proposition \ref{prop::grad-converge}. Namely, the approximate gradient can converge to the expectation of the profit gradient gradient under some mild conditions. The approximate gradient allows us to design gradient-based algorithms based on the approximate gradient. In our work, we develop the AGD algorithm based on the gradient descent algorithm. The concept of approximate gradient solves the difficulty when gradient information is inaccessible for the true model and the approximate model is intractable. 

3. To validate our algorithm, we first prove the convergency result of AGD algorithm. We show that any converging subsequences generated by AGD algorithm converges to points that have a bounded expected gradient, which is a necessary condition of optimality (Theorem \ref{theo::nece-cond}). We also investigate the convergence results for the price-only and price-adjustment cost models. We prove that the objective error is bounded under general convex conditions (Theorem \ref{theo::convergence-convex}), which corresponds to the price-only model without price-adjustment cost. The solution sequence generated by the algorithm converges to the true optimal solution if the strongly convex condition holds (Theorem \ref{theo::strconv-dist-stable}), which corresponds to the price-adjustment cost case. 

4. We gain insights from the numerical experiment on the importance of the distribution-free and decision-dependent properties in our model. Although these two properties increase the complexity of the model, ignoring the decision-dependent property will result in unreasonable predictions and solutions. At the same time, when there is a lack of information about the relationship between demand and price, the use of distribution-free model can adapt to complex unknown demand distributions.

The rest of the paper is organized as follows. Some relevant works are reviewed in Section \ref{sec::literature}. Section \ref{sec::model-description} outlines our approach for approximating the model and introduces the concept of approximate gradients and the approximate gradient descent algorithm. It also discusses the impact of decision-dependent effects on the resulting approximate model. In Section \ref{sec::analysis}, we analyze the convergence of the AGD algorithm under non-convex, convex, and strongly convex conditions and connect these results to the newsvendor pricing model and its extended models. The experiment result on AGD algorithm with managerial insights on the distribution-free and decision-dependent properties are provided in Section \ref{sec::experiment}. Section \ref{sec::conclusion} concludes our work and outlines future research directions. For the briefness of reading, we furnish all the proofs in the E-companions.

\subsection{Relevant Literature}\label{sec::literature}

In this part, we summarize some relevant streams in the literature, including \textit{newsvendor pricing problem}, \textit{Distribution-free newsvendor problems} and \textit{decision affects uncertainty models}. 

\textit{Newsvendor pricing problem}. The newsvendor pricing problem has been studied for decades. A systematic review on this topic can be seen in \cite{price-only-Deyong2019}. One of the key aspect in this problem is the relationship between demand and price. There are two forms of stochastic demand: additive and multiplicative. \cite{price-only-Dada1999} gave the theoretical optimality form in both two demand scenarios. Apart from the demand uncertainty, \cite{two-side-9-price-setting} and \cite{two-side-10-price-setting} studied the pricing problem under uncertain supply. And \cite{risk-averse-msom1} further investigated the risk-averse effect on the pricing decision. However, the research above relied on the assumption of a concrete demand distribution form, while we focus on distribution-free scenarios and do not impose any assumption on the form of demand distribution and its relationship to price. \cite{ijoo-price-setting-2021} proposed an ML framework to solve the distribution-free newsvendor pricing problem based on regression methods. But their work was limited to solving the newsvendor pricing problem. In this work, we focus on a more general framework that can solve decision-dependent problems of the same category as the newsvendor pricing problem. 

There are also some variants of the newsvendor pricing problem. \cite{price-only-Deyong2019} mentioned the price-only newsvendor model that only makes pricing decisions. The model with cost of price adjustment is another variant. In the inventory and pricing coordination literatures, the price-adjustment cost is usually employed to the multi-period inventory management problems \citep{adjustment-cost7-multiperiod-Lu2018,adjustment-cost8-multiperiod-Chen2015}. Although we do not consider the costs in other periods in the single-period newsvendor problems, there exists the ``reference price'' that plays a similar role as previous price \citep{adjustment-cost9-referpri-chen2016,adjustment-cost10-referpri-Srivastava2022}. Results from \cite{adjustment-cost11-revenue-Sabri2009} also showed that the change of price can affect the revenue management of perishable products. Our work investigates the extend price-only model with and without price-adjustment cost and address the nonlinearity brought by the price-adjustment cost.

\textit{Distribution-free newsvendor}. Many recent studies have utilized data-driven approaches to address distribution-free newsvendor models and newsvendor pricing problems. Earlier studies estimated demand based on historical samples before solving the order decision under forecasted demand. An typical approach is the SAA. This method is to calculate the average profit under a certain decision by replacing the random variable with the demand samples \citep{SAA-Kleywegt,SAA-Mello2001}. \cite{SAA-Feng} pointed out that a pure SAA approach will cause overfitting. Therefore, common approaches including adding a regularizer or constraints that can control the predicted profit variability \citep{no-feature-Levi-1,no-feature-Levi-2,SAA-3-Cheung-2019,SAA-4-Qin-2022}. The SAA approach cannot directly suit our problem because the decision variable can affect the distribution of random parameters. We cannot predict the demand by simply replacing the random parameters with historical data, since the historical decision is not equal to the decision at this moment. 

The empirical risk minimization (ERM) is a promoted approach compared with the SAA method that ignores the feature information. The ERM method aims to find a function that directly maps the observed feature to optimal decision \citep{ERM1-Vapnik1998}. \cite{OR-big-data-newsvendor} highlighted the importance of demand features for decision making, and they demonstrated that decision made without features is biased. They focused on the standard newsvendor model where the supply is deterministic and no risk-averse strategy is adopted, while we further consider the price decision that can affect demand. However, the ERM method cannot handle the decision-dependent property either, since it still assumes that the demand distribution is independent of price decision. 

\cite{Bertsimas-2019} proposed a general prescriptive approach to solve the prescribe optimization problem. They estimated the objective using ML methods. The input of such ML methods are features, and decision variables if the decision can affect uncertainty. The output of the ML methods are treated as the sample weight, which is used to directly estimate the true objective. We adopt the same concept of weight approximation as their work. But their work bypassed the solution approach of such model in the continuous settings, which must be addressed in our work. \cite{sol-Bertsimas} further extended the work by adding the penalize term into the objective. But they only provided solution methods on tree-based weights. And it is still not clear how to handle other types of weights. In our work, we give a general solution approach to such problems that is suitable for any kind of approximation. In \cite{Shen-2022-POM}, this prescriptive model was extended to constrained conditions. They focused on the traditional newsvendor problem with profit risk constraint and contextual information. However, their work had fundamental differences with our model since they did not consider the case when decision affects uncertainty. In concrete, the supply decision would not affect the stochastic demand in traditional newsvendor problem that \cite{Shen-2022-POM} analyzed. 

Other data-driven approaches related to the newsvendor problem include the quantile regression \citep{ijoo-price-setting-2021}, deep neural network \citep{deep-learning-Davood, deep-learning-zhang} and robust optimization (RO) \citep{RO4-msom}. Compared with these works, we study a distribution-free pricing problem with decision-dependent effect, and extend the model to a convex price-only case and nonlinear price adjustment cost case. We manage to solve the complex decision-dependent model by proposing the concept of approximate gradient.

\textit{Decision affects uncertainty models}. Though the decision-dependent effect is not fully researched in the field of data-driven management, several stochastic programming works have investigated the solution approaches to the optimization problems under this effect. Some studies addressed the problem by assuming the form of distribution. For instance, \cite{DAU2-rev-Hellemo2018} explored several ways to solve the decision-dependent models. \cite{DAU3-LLR-Liujy2021} developed an algorithm based on local linear regression (LLR) models. In concrete, they used LLR models to approximate the distribution in the neighborhood of a decision point. The models above relied on some assumptions on the form of the distributions. Instead, our work do not require any assumption on the distribution rule. There are also works that did not rely on the distribution assumptions. \cite{DAU1-Dimitris2022} proposed a sampling method to solve the decision-dependent models when the distribution is under Lipschitz control. But their approach cannot be directly used in our model since we cannot access the demand distribution on every decision points, thus the sampling process cannot be carried out. 

Another stream of work that considers the decision-dependent effect is the RO / distributionally robust optimization (DRO). \cite{RO3-Luo2020} constructed the decision-dependent robust set based on the distribution moment, and \cite{RO1-Xiong2022} constructed a sorting-dependent feedstock condition ambiguity set to solve a stochastic resource allocation problem. But their models relied on the knowledge of the distribution family. \cite{RO2-WDistance-Noyan2021} constructed the ambiguity set by Wasserstein distance, but they assume that the distribution map is known in advance, while we treat the distribution map as an unknown function and do not impose any assumption on the form of the distribution map.

Our solution approach refers to the repeated gradient descent approaches proposed by \cite{DAU4-Perdomo2020}, where the descent direction is determined through the expectation of gradient. We extend their work to the distribution-free scenario and illustrate how to approximate the expectation of gradient when the underlying distribution is unknown. Moreover, they only investigated the performance under strongly-convex condition, while we extend the performance analysis of the AGD algorithm to general cases. 

% Note that our solution approach is not limited to the newsvendor pricing problem, but it can be used in other models with decision-dependent probabilities since we do not impose any assumption on the form of objective function or distribution. This versatility further motivates its use.

In summary, this study considers several real-world costs and effects in the newsvendor pricing problem. We use the data-driven prescriptive method to construct our approximate model, but the decision-dependent effect significantly increase the complexity of approximate model and makes the gradient of the true model intractable. We develop a new approach that integrates ML approximation and stochastic programming. The solution of our model brings insights to other decision-dependent programming problems.

%% file: model.tex
\section{Model Description}\label{sec::model-description}

\subsection{Model and Approximation}
\label{sec::model}

We first construct the single-product newsvendor price-setting problem. We denote $c$ and $s$ as the purchase cost and salvage value, respectively. Given a new scenario with feature $Z=z$, the manager should determine the selling price $p$ and order quantity $q$ at the same time. We assume $s<c\leq p$, and denote the dimension of a vector $v$ as $dim(v)$. The total cost associated with the decision and random demand $D$ is:

\begin{equation}
	\label{eq::def-pi-noadj}
	l(p,q,D) = -p(D\wedge q) + cq - s(q-D)^+
\end{equation}

where $x\wedge y = \min\{x, y\}$, and $(x)^+ = \max\{x, 0\}$. 

% Now we consider the price-adjustment cost. We denote the general price-adjustment cost with reference price $\tilde{p}$ as $\Delta(p;\tilde{p})$. Several forms of the price-adjustment cost have been used in literature \citep{adjustment-cost2-Chen2011,adjustment-cost6-CHEN2018}. We adopt the quadratic form $\Delta(p;\tilde{p}) = \gamma(p-\tilde{p})^2$ as \cite{adjustment-cost4-Rotemberg1982,adjustment-cost3-Roberts1992}. Note that the choice of the price-adjustment function does not matter as long as the adjustment cost is convex on $p$, because our solution approach is only regard to the convexity of objective function. The profit function with price-adjustment cost is:

% \begin{equation}
% 	\label{eq::def-pi-adj}
% 	\pi_a(p,q,D) = p(D\wedge q) - cq + s(q-D)^+ - \gamma (p-\tilde{p})^2
% \end{equation}

% We use the definition of $\pi_a(p,q,D)$ there to differentiate the price-adjustment model with the origin model. If the price-adjustment is not considered, the terms in the expectation denotation can be replaced by $\pi(p,q,D)$ defined in (\ref{eq::def-pi-noadj}). In the following passage, the profit function is $\pi(p,q,D)$ if not mentioned. 

We then substitute the cost function into model (\ref{eq::true-obj-pi}) and get the true newsvendor pricing model:

\begin{align}
	\mbox{(\textbf{True Model})}&\min_{p,q}f(p,q;z)\overset{def}{=}\mathbb{E}_{D\sim f_D(p,z)}[-p(D\wedge q) + cq - s(q-D)^+], \label{eq::true-obj-expand}
\end{align}

Note that the objective function in (\ref{eq::true-obj-expand}) is decision-dependent, which implies that the distribution $f_D(p,z)$ shifts correspondingly once the price decision changes.

In the data-driven problem, the distribution of $D$ is unknown for any specific $p$ and $z$ and only data $S_N = \{(p^1,z^1,D^1),...,(p^N,z^N,D^N)\}$ is available. Let $(p^*,q^*)$ denotes the full-information optimal decision, which maximize the decision-dependent objective function while satisfying risk-averse constraint. The prescriptive approximation is to use $S_N$ to construct a data-driven approximation to (\ref{eq::true-obj-expand}). We approximate the model through the weighted sample approach as \cite{Bertsimas-2019}. Consider the approximate model of the form

\begin{align}
	\mbox{(\textbf{Appr-Model})}\min_{p,q}\hat f(p,q;z)\overset{def}{=}\sum_{i=1}^Nw^i(p,z)l(p,q,D^i), \label{eq::appr-obj}
\end{align}

where $w^i(p,z)$ are weight functions derived from the data by ML methods. Our approximate model (\ref{eq::appr-obj}) does not restrict the form of the weight functions. For briefness, we only present the definition of kNN weight below. Other definitions of weight functions (e.g. kernel regression (KR), classification and regression tree (CART), random forest (RF)) can be seen in Section \ref{sec::appd-weight-def}. Readers can also refer these weight definitions to \cite{Bertsimas-2019}.

\begin{definition}[kNN weight]
	The weight function can be derived from the definition of kNN: 

	\begin{equation}
		\label{eq::weight-KNN}
		w^{\mbox{kNN},i}(p,z)=\frac 1 k \mathbb{I}\{(p,z)\mbox{ is a kNN of } (p^i,z^i)\}, \quad \forall i \in [N]
	\end{equation}

	where $\mathbb{I}\{\cdot\}$ is the indicator function, $[N]=\{1,2,...,N\}$ denotes the index set and $x^i$ is a kNN of $x$ if and only if $|\{j\in\{1,...,N\}\backslash i:dist(x^j,x)<dist(x^i,x)\}|<k$. We use the Euclidean distance $dist(x,y)=\Vert x-y\Vert_2$ here to represent the distance between two vectors. The kNN weight function indicates that all sample points that are among the k-nearest neighbor share the same weight, and points that are not included are not taken into consideration. Finding kNN points of $(p, z)$ can be done in $O(N*(dim(z)+1))$ time and can be sped up by using the kD tree. 
\end{definition}

We note that the weight approximation has practical meanings. It reflects the similarity between the previous condition and the current condition. The condition is variables that can affect random parameters in the model. In our newsvendor pricing problem, it means the pricing decision and features that can affect the random demands. For instance, when the distance between current condition $(p,z)$ and previous condition $(p^i,z^i)$ is large, the kernel weight will decrease, meaning that sample $i$ is not similar to current condition and plays a minor role in aiding the decision making. 

% 模型的渐进最优性

Specifically, when we take kNN weight function, the approximate model is equivalent to a nonlinear mixed integer programming (NMIP) problem in (\ref{eq::NMIP-model}). This equivalent model can exemplify the intractability incurred by the decision-dependent effect.

\begin{equation}
	\label{eq::NMIP-model}
\begin{aligned}
	\min_{p,q,x_i,\beta_i,y_i} &\sum_{i=1}^Nx_i\left[-(p-c)q+(p-s)y_i\right]\\
	\mbox{s.t. } & y_i\ge q-D^i,\quad\forall i \in [N]\\
	&\sum_{i=1}^N x_i=k\\
	&(p-p^i)^2-(p-p^j)^2+(\Vert z-z^i\Vert^2-\Vert z-z^j\Vert^2)\\
	&\leq M(x_j-x_i+1), \quad\forall i\in[N],j\in[N], i\not=j\\
	&p,q\ge 0, x_i\in\{0,1\}, y_i\ge 0,\quad \forall i \in [N]
\end{aligned}
\end{equation}

where $x_i$ are auxiliary variables that indicates whether sample $i$ is a kNN of current condition $(p,z)$. $y_i$ is the auxiliary variable representing the max operator. $M$ is a sufficiently large positive number. Remark that the decision variable $x_i$ indeed reflects the decision-dependent characteristic since $x_i$ will degenerate to a weight parameter if the problem is not decision-dependent. We observe that the nonlinearity is mainly cost by $x_i$, indicating the decision-dependent characteristic significantly increases the complexity of the model. 

% 近似模型非常难解
We can see from the equivalent model that the approximate model is difficult to solve both in theory and in practice. It demonstrates that solving (\ref{eq::appr-obj}) is at least as difficult as optimizing a NMIP model, which is NP-hard. Indeed, for a fixed feature $z$, $\hat f(\cdot,\cdot;z)$ is nonlinear and not even continuous since the decision variable $p$ is in the kNN weight weight function. We are therefore motivated to develop a reasonable approach for solving the decision-dependent model.

% In the price-setting problem, we determine the best price and order quantity to maximize the total profit: 

% \begin{equation}
% 	\max_{p,q}\mathbb{E}_{D,Q}[\pi(p,q;D,Q)|p,q]
% \end{equation}

% where $p$  and $q$ are the price and order quantity respectively. $D$ is the uncertain demand that is determined by $p$, and $Q$ is the uncertain supply, whose distribution depends on $q$. And the profit function can be written as 

% \begin{equation}
% 	\label{eq:def-pi}
% 	\pi(p,q;D,Q) = p (D\wedge Q) - cQ + s (Q-D)^+
% \end{equation}

% where $x\wedge y = \min\{x,y\}$ and $(\cdot)^+=\max\{\cdot,0\}$. We denote $c$ and $s$ as the unit order cost and salvage value of the product respectively. 

\subsection{Approximate Gradient}

% 对目标函数求导

In this section, we focus on deriving a tractable method to solve the decision-dependent model. Our approximate function can be derived from the partial gradient of the cost function. Ideally, when the solution sequence moves to the direction of our approximate gradient, the generated sequence will converge to a stationary point as in the gradient descent method. When the cost function is convex, the corresponding sequence will get close to the optimal solution.

To begin with the derivation of the approximate gradient, we first derive the gradient of cost function. Note that since the cost function is not smooth but concave in $q$, we derive one of the subgradient. We use $\nabla$ to denote the gradient and $\partial$ to denote the subgradient.

\begin{equation}
\label{eq::profit-grad}
	\partial_{p,q} l(p,q,D)= \left\{\left[
		\begin{aligned}
			&-(D\wedge q)\\
			& -(p-c) + (p-s)e
		\end{aligned}
	\right]: e\in \left[\mathbb{I}\{q>D\}, \mathbb{I}\{q\ge D\}\right]\right\}
\end{equation}

Though the denotation $\partial_{p,q} l(p,q,D)$ usually refers to the subgradient set of the cost function, to simplify the denotation, we use it to denote any elements belongs to the subgradient set in the following passage. In practice, the choice of element only affects the approximate gradient at $q=D$.  We take the same approximate approach as model (\ref{eq::appr-obj}) , which we formally state in Definition \ref{def::appr-grad}.

\begin{definition}[Approximate Gradient]
	\label{def::appr-grad}
Given a feature $z$ and a decision point $(p, q)$, the approximate gradient at this point is defined as 

\begin{equation}
	\label{eq::appr-grad}
	G^N(p,q;z) \overset{def}{=}\sum_{i=1}^Nw^i(p,q,z)\partial_{p,q}l(p,q,D^i) 
\end{equation}

\end{definition}

Note that in newsvendor pricing model, the approximate gradient is actually an approximation to all elements in the subgradient set. Next, we state the following proposition, which formally shows that the approximate gradient converge to the expectation of cost gradient.

\begin{proposition}
	\label{prop::grad-converge}
	Suppose the joint distribution of $z, p, q$ is absolutely continuous and has density bounded away from 0 to $\infty$ on the support of $p,q,z$ and twice continuously differentiable. And suppose the dataset $\{p^i,q^i,D^i\}$ comes from an iid process. Then for arbitrary $z$, we have 
	\begin{equation*}
		\lim_{N\rightarrow +\infty}\sup_{p,q}\left\Vert G^N(p,q,z) - \mathbb{E}_{D,Q\sim f_{D}(p,z)}[\partial_{p,q}l(p,q,D)]\right\Vert = 0
	\end{equation*}
\end{proposition}

Proposition \ref{prop::grad-converge} shows the convergency of our approximate approach. We can then develop gradient-based algorithms by this approximate gradient. But before proceeding to the application of the approximate gradient, we should remark the difference between the following two concepts: the \textit{expectation of cost gradient} and \textit{gradient of objective expectation}. This two concepts are different because of the decision-dependent effect.

\begin{proposition}
	\label{prop::fail-converge-grad}
	Generally, the expectation of cost gradient $\mathbb{E}_{D\sim f_{D}(p,z)}[\partial_{p,q}l(p,q,D)]$ is not equal to the gradient of objective expectation $\partial_{p,q}\mathbb{E}_{D\sim f_{D}(p,z)}[l(p,q,D)]$. And $G^N(p,q,z)$ fails to converge to the subgradient of objective expectation.
\end{proposition}

Proposition \ref{prop::fail-converge-grad} indicates that, unfortunately, our approximate gradient cannot converge to the true gradient of the objective function. But in the next section, we will show that (1) the stationary point of the expectation of gradient is a necessary condition for the optimal solution, and (2) if the cost function is convex or strong convex, the approximate gradient ascent method can generate solution sequence with bounded error.

Now we list the AGD algorithm in Algorithm \ref{algo::approx-frand-wolfe}. Note that the step sizes also need confirmation in order to implement AGD algorithm. We adopt the diminishing step size and Armijo step size in our work. The analysis on the selection of step size can be seen in the following section and numerical experiment.

\begin{algorithm}[!ht]
	\renewcommand{\algorithmicrequire}{\textbf{Input:}}
	\renewcommand{\algorithmicensure}{\textbf{Output:}}
	\caption{Approximate Gradient Descent Algorithm}
	\label{algo::approx-frand-wolfe}
	\begin{algorithmic}[1]
		\Require initial solution $p^0,q^0$, step size $\eta^r$, dataset $\{p^i,z^i,D^i\}_N$;
		\Ensure solution $\hat p^*, q^*$;
		\State $r=0$;
		\While {Stop criteria not satisfied}
			\State Calculate approximate subgradient $G^N(p^r_N,q^r_N;z)$ by (\ref{eq::appr-grad});
			\State Calculate $(p^{r+1}_N,q^{r+1}_N)=\left((p^r_N,q^r_N)-\eta^rG^N(p^r_N,q^r_N;z)\right)^+$;
			\State $r = r + 1$;
		\EndWhile
		\State\Return $p^r_N,q^r_N$;
	\end{algorithmic}
\end{algorithm}

%% file: analysis.tex
\section{Convergence Analysis of the AGD Algorithm}\label{sec::analysis}

% TODO: 用corollary 将凸结论和newsvendor 联系起来，思考非凸非连续的时候怎么办

In this section, we investigate the convergence of the sequence ${p^r,q^r}$ generated by the AGD algorithm. Intuitively, the convergence results are similar to the typical gradient (or subgradient) methods if the number of samples is sufficient, since we have proved in Proposition \ref{prop::grad-converge} that the approximate gradient is close to the expectation of gradient. The convergence analysis is comprised of three parts. In Section \ref{sec::preliminaries}, we state some denotations and prerequisite assumptions. In Section \ref{sec::convergence-nonconvex}, we first prove the convergence result of AGD algorithm with both diminishing and armijo step size. This result does not require the convexity of cost function, thus can be used to solve our newsvendor pricing model (\ref{eq::true-obj-pi}). We also investigate the price-only and price-adjustment cost models which is convex and strongly-convex and show the error bound in both conditions in Section \ref{sec::convergence-convex}. 

\subsection{Preliminaries}\label{sec::preliminaries}

We first introduce some denotations and assumptions for the following analysis. To simplify the denotation, we use $x = (p, q)$ to denote all decision variables In the following analysis. Thus the cost function becomes $l(x,D)$. The objective is to minimize $\E{D}{x}{l}$, where $f_D(x,D)$ denotes the probability density function of $D$. 
% We can associate $\pi(x,D)$ in the newsvendor model with $l(x,D)$ in the generalized model, and the probability density function of demand $f_D$ with $f_D(x,D)$. Remark that the newsvendor model addresses a maximization problem, whereas the generalization model deals with a minimization problem. Therefore, the convexity or strong convexity in the generalized model in the following text corresponds to the concavity or strong concavity of the newsvendor model.

We then introduce some assumptions that are used in both nonconvex and convex occasions. 

\begin{assumption}[Defferentiate-integrate exchange]\label{assum::switch-int}
	We assume that $l(x,D)$ is differentiable in $x$, and the gradient is bounded by an $L^1(D)$ function $g(D)$. That is, $\exists g\in L^1$, $\Vert \nabla_xl(x,D)\Vert\leq g(D)$ for all $x,D$. 
\end{assumption}

where $L^1(D)$ denotes the set of functions that are integrable in $D$ almost everywhere \citep{real-analysis-Folland}. Assumption \ref{assum::switch-int} implicates that we can change the order of integration and derivation when calculating the derivative of the integral of $l(x,D)$ (see Theorem 2.27 in \cite{real-analysis-Folland}). That is, 

\begin{equation*}
	\nabla_x\int_Dl(x,D)f_D(x,D)dv = \int_D\nabla_x (l(x,D)f_D(x,D))dv,
\end{equation*}

which enables us to access the derivative of objective function. This assumption is reasonable since most cost functions are integrable almost everywhere in practice. In the following text, we assume that the cost function satisfies Assumption \ref{assum::switch-int}.

We also assume that the distance between the decision-dependent distributions under different decisions can be bounded by the distance between the two decisions.

\begin{assumption}[$\epsilon$-sensitivity]\label{assum::eps-sensitive}
	We assume that the distribution map $f_D(\cdot,D)$ is $\epsilon$-sensitive. That is, for all $x_1, x_2\in X$, 

	\begin{equation}\label{eq::eps-sensitive}
		W_1(f_D(x_1,D), f_D(x_2,D))\leq\epsilon\Vert x_1-x_2\Vert_2
	\end{equation}

	where $W_1$ denotes the earth mover's distance \citep{earth-mover-distance}.
\end{assumption}

This assumption has been mentioned in \cite{DAU5-Perdomo2020}. Intuitively, Assumption \ref{assum::eps-sensitive} ensures that the difference between decision-dependent distributions is not too large under different decisions. Therefore, when analyzing the gap between the approximate solution and the optimal solution, we can convert the difference between expectations under different distributions into the distance between decision variables. 

\subsection{Convergence Under Nonconvex Condition}\label{sec::convergence-nonconvex}

In this section, we focus on the convergence result under general cases. We first establish the convergence to a stationary point of the cost gradient expectation for diminishing step size, and propose that the converging points of armijo step size have bounded gradient expectation. We then prove a necessary condition for optimality, which connects the two convergence results with the optimality property. Finally, we extend the convergence result to the non-smooth newsvendor pricing model.

We first introduce some assumptions that is used to prove the nonconvex convergence results. 

\begin{assumption}[Same limited range]\label{assum::same-limit-range}
The value range of random parameter $D$ remains the same under any $x$. And the value range of $D$ is limited. 
\end{assumption}

Assumption \ref{assum::same-limit-range} can be satisfied in practical settings. We can take the union set of the value ranges of $D$ under different $x$ and assign the probability outside of the distribution $D|x$ as $0$. In practice, the demand is positive and we can often obtain a maximum demand, thus Assumption \ref{assum::same-limit-range} holds. We denote the range and its volume as $\Omega$ and $S_\Omega$.

\begin{assumption}[Lipschitz continuous and Lipschitz gradient]\label{assum::lip}
The cost function $l(x,D)$ is smooth, Lipschitz continuous with Lipschitz gradient. And the probability density function (pdf) of random parameter $D$ has Lipschitz gradient. That is, for any $x,y\in X$, we have
\begin{flalign*}
&\ (a) \frac{|l(x,D)-l(y,D)|}{|x-y|}\leq L_1, \quad\forall D\in \Omega &\\
&\ (b) \frac{\Vert \nabla_x l(x,D)-\nabla_x l(y,D) \Vert}{|x-y|}\leq L_2, \quad\forall D\in\Omega &\\
&\ (c) \frac{\Vert \nabla_x f_D(x, D)-\nabla_x f_D(y,D) \Vert}{|x-y|}\leq L_3, \quad\forall D\in\Omega&
\end{flalign*}
\end{assumption}

Assumption \ref{assum::lip} is reasonable since the value ranges of decision variables are usually bounded in practice. In the newsvendor pricing problem, the cost function $l(p, q, D)$ is also piecewise linear in $p$ and $q$ separately. We also assume that the distribution function of random parameters has a limited change rate when the decision $x$ changes. 

\begin{assumption}[Bounded cost and distribution gradient]\label{assum::bound}
The absolute value of cost function and the gradient of the pdf of random parameter is bounded. In words, for any $x\in X$ and $D\in\Omega$, $|l(x,D)|\leq L_4$ and $\Vert\nabla_x f(x, D)\Vert\leq L_5$. 
\end{assumption}

Assumption \ref{assum::bound} can also be satisfied in practice since the cost is limit given a reasonable decision range. We also assume that the distribution change is smooth in $x$. In our newsvendor pricing model, the mild change in price decision will not cause sudden change in demand.

Now we first establish the theoretical convergence result of AGD under nonconvex and smooth conditions. 

\begin{proposition}[Convergence under diminishing step size]
	\label{prop::converge-stationary}
	Under assumptions \ref{assum::switch-int} - \ref{assum::bound}, if the cost function $l(x, D)$ is twice differentiable in $x$ and the step size $\eta^r$ is diminishing with $\sum_{r=0}^\infty \eta^r=\infty$, the sample size $N$ is sufficiently large and the value range volume $S_\Omega$ is sufficiently small, then any limit point of the sequence generated by AGD algorithm is a stationary point of the cost gradient expectation.
	
	\begin{equation}\label{eq::converge-stationary}
	\mbox{if } \lim_{N\rightarrow\infty}\lim_{r(\in\mathcal{K})\rightarrow\infty}x^r_N=\bar x, \mbox{ then }\mathbb{E}_{D\sim f_D(\bar x, D)} [\nabla_x l(\bar x, D)] = 0
	\end{equation}
\end{proposition}

Proposition \ref{prop::converge-stationary} provides that the AGD algorithm converges to zero points of expected gradient. Notice that Proposition \ref{prop::converge-stationary} relies on Assumptions \ref{assum::same-limit-range} and \ref{assum::bound}, which are strong conditions and may cause the algorithm fails to provide a good convergence rate. When the constants in these assumptions, such as $\Omega, L_4, L_5$ is large, the convergence will fail. Therefore, we then state the convergence result under armijo step size which require a milder condition.

\begin{proposition}[Convergence under armijo step size]
	\label{prop::converge-stationary-armijo}
	Under Assumptions \ref{assum::switch-int}, \ref{assum::eps-sensitive} and \ref{assum::lip}(a), suppose the AGD algorithm adopts armijo step size with $\sigma$, and the sample size $N$ is sufficiently large, then any limit point $\bar x$ of the sequence generated by AGD algorithm has a bounded gradient expectation.

	\begin{equation}\label{eq::converge-stationary-armijo}
		\mbox{if } \lim_{N\rightarrow\infty}\lim_{r(\in\mathcal{K})\rightarrow\infty}x^r_N=\bar x, \mbox{ then }\Vert\mathbb{E}_{D\sim f_D(\bar x, D)} [\nabla_x l(\bar x, D)]\Vert \leq \frac{\epsilon L_1}{1-\sigma}
	\end{equation}

\end{proposition}

Unlike diminishing step, the expected gradient for armijo step is not guaranteed to converge to $0$, but its convergence can hold under a milder condition, where Assumptions \ref{assum::same-limit-range} and \ref{assum::bound} may not hold. We will see in Section \ref{sec::exp-convergence-perform} that the armijo step size is more useful than the diminishing step size in most practical cases.

The convergence results above are both relevant to the expected gradient. However, the expected gradient does not equal to the gradient of objective according to Proposition \ref{prop::fail-converge-grad}. In other words, the converged points are not the stationary point of objective function. So we need to further investigate the relationship between the converged point and optimal solution.

\begin{theorem}[Necessary condition of optimality]\label{theo::nece-cond}
If $x^*$ is the optimal solution for a decision-dependent problem $\min_{x} \E{D}{x}{l}$, where $l$ is an $L_1$ Lipschitz function and Assumptions \ref{assum::switch-int} and \ref{assum::eps-sensitive} are satisfied, then $\Vert \Enabla{D}{x^*}{l} \Vert \leq L_1\epsilon$.
\end{theorem}

Theorem \ref{theo::nece-cond} builds connection between the AGD algorithm and optimality. It indicates that one necessary condition for optimality is that the norm of the expected gradient should not be too large. For diminishing step size, the expected gradient will be sufficiently small and thus satisfies the necessary condition. For armijo step size, we notice that the necessary bound in Theorem \ref{theo::nece-cond} is actually the upper bound when $\sigma=0$ in Proposition \ref{prop::converge-stationary-armijo}. Therefore, the AGD algorithm can satisfy the necessary condition of optimality. 

For the newsvendor pricing problem, the convergence results are not strictly held since the cost function $l(x, D)$ is non-smooth on $q$. But we can also apply the AGD algorithm by selecting an element from the subgradient set. And we can also get a good convergence performance since the convergence results still holds when $q^r$ is in any side of $D$. The performance of AGD algorithm on such problem is further validated numerically in Section \ref{sec::exp-convergence-perform}.

\begin{remark}[Generality of the AGD algorithm and the convergence results]
	
	It is important to note that the AGD algorithm and the convergence results are not limited to the newsvendor pricing problem. It can suit for other unconstrained models with decision-dependent property since we do not impose any assumption on the form of objective function or distribution. It also does not rely on the approximation approach. This versatility further motivates its use. 
\end{remark}

\subsection{Convergence Under Convex Condition}\label{sec::convergence-convex}

In this section, we provide theoretical guarantee on the error bound of AGD algorithm when the cost function is convex and strongly convex.  
% The convergence result can suit for several practical newsvendor pricing variants, such as the price-only model and price-adjustment cost. Except for the newsvendor pricing models, the convergence results also hold for other decision dependent models. 
In the previous section, we conclude that the subsequences given by AGD algorithm satisfy the necessary condition of optimality. However, the optimality cannot be guaranteed in the general case. Therefore, we focus on the convex condition in this section. The goal of this section is, thus, to investigate the theoretical error bound and convergence to the optimal solution in the convex and strongly convex conditions. We also introduce two variants of newsvendor pricing model to connect the convergence analysis to practical scenarios. The analysis on the convex cases can be applied to the price-only model, and the results under the strongly-convex condition can be applied to the price-only model with price-adjustment cost.

We start with some definitions and assumptions. The price-only model shares the same objective function as the pricing model (\ref{eq::true-obj-expand}), but the order quantity $q$ is not a decision variable \citep{price-only-Dada1999}. In this case, the cost function is concave on the decision variable $p$. The price-only model can be written as 

\begin{equation}
	\label{eq::price-only}
	\min_{p} f_p(p;z,q) = \mathbb{E}_{D\sim f_D(p,z)}[l(p,q,D)]
\end{equation}

The price-only model is still decision-dependent since the demand distribution is affected by the price decision. Note that solving the price-only model is also helpful for solving the pricing problem with order quantity decision. Since the manager can choose among several reasonable order quantities and decide the optimal prices respectively \citep{price-only-Deyong2019}.

We are also interested in the price-adjustment cost in practice, which is defined by the absolute deviation to the original price $\tilde{p}$. That is, 

\begin{equation}
	\label{eq::price-adjustment}
	\min_p f_a(p;z,q) = \mathbb{E}_{D\sim f_D(p,z)}[l_a(p,q,D)]
\end{equation}

where 

\begin{equation}\label{eq::def-pi-adj}
	l_a(p,q,D) = l(p,q,D) + \gamma (p-\tilde{p})^2
\end{equation}

We denote $\tilde{p}$ as the reference price. In our work, we adopt the quadratic form $\Delta(p;\tilde{p}) = \gamma(p-\tilde{p})^2$ as \cite{adjustment-cost4-Rotemberg1982,adjustment-cost3-Roberts1992}. But the choice of the price-adjustment function does not matter as long as the adjustment cost is convex on $p$, because our solution approach is only regard to the convexity of objective function. We also note that although both price-only model and price-only model with price-adjustment cost are single-variable problem, our algorithm and analysis can suit for multi-dimension decisions. Similar to Section \ref{sec::convergence-nonconvex}, we still use the cost function $l(x,D)$ to perform the analysis. 

To connect the sequence with optimal solution, we need a intermediate \textit{stable point} that satisfies

\begin{equation}
	\label{eq::stable-point}
	x_{PS} = \arg\min_{x} \mathbb{E}_{D\sim f_D(x_{PS},D)}[l(x,D)]
\end{equation}

In \cite{DAU4-Perdomo2020}, the stable point serves as a fix point of the updating rule $x^{r+1} := \arg\min_{x}\mathbb{E}_{D\sim f_D(x^r,D)}[l(x,D)]$. They prove that the solution sequence will converge to $x_{PS}$ if the descent direction is the true expected gradient $\Enabla{D}{x}{l}$ and the cost function $l(\cdot,D)$ is strongly convex. In our work, however, the expected gradient is unknown. We prove that our approximate gradient can also converge to the stable point under some conditions. 

We derive performance bounds on convex and strongly convex cases under the following assumptions.

\begin{assumption}[Lipschitz gradient]\label{assum::conv-lip-grad}
$l(x,D)$ has $L^c_1$-Lipschitz gradient both in $x$ and $D$. In words, there exists $L^c_1$ such that
\begin{align*}
	\Vert \nabla_x l(x_1, D)-\nabla_x l(x_2, D)\Vert\leq L^c_1\Vert x_1-x_2\Vert, \quad\forall x_1,x_2,D \\
	\Vert \nabla_x l(x, D_1)-\nabla_x l(x, D_2)\Vert\leq L^c_1\Vert D_1- D_2\Vert,\quad\forall x,D_1,D_2
\end{align*}
\end{assumption}

\begin{assumption}[Lipschitz continuous and bounded gradient]\label{assum::conv-lip-cont}
$l(x,D)$ is $L^c_3$ Lipschitz continuous in $D$ and the gradient of $L$ is bounded. In words, there exists $L^c_2,L^c_3$ such that 

\begin{align*}
	\Vert  l(x, D_1)- l(x, D_2)\Vert\leq L^c_2\Vert D_1-D_2\Vert, \quad\forall x_1,x_2,D \\
	\Vert \nabla_x l(x,D)\Vert \leq L^c_3 ,\quad\forall x,D
\end{align*}
\end{assumption}

Assumptions \ref{assum::conv-lip-grad} and \ref{assum::conv-lip-cont} can hold for both the price-only models with and without price-adjustment cost. Although there is quadratic term in the cost function with price-adjustment cost, which means Assumption \ref{assum::conv-lip-cont} may not hold when $p$ is unconstrained. In practice, the price decision is usually within a reasonable range and thus \ref{assum::conv-lip-cont} can be satisfied in most cases. 

Based on the assumptions above mentioned, we then explore the theoretical guarantee of the convex and strongly convex cases, respectively.

\subsubsection*{Convex case: the price-only model}

Theorem \ref{theo::convergence-convex} states the error bound in the convex case.

\begin{theorem}[Objective Error in Convex Case]
	\label{theo::convergence-convex}
	Suppose that Assumptions \ref{assum::switch-int}, \ref{assum::eps-sensitive} and \ref{assum::conv-lip-cont} are satisfied. Denote $\lim_{N\rightarrow \infty}x^r_N=x^r$, and the optimal solution is $x^*$. If $l(x,D)$ is convex, after $k$ iterations, 

	\begin{align*}
		\label{eq::converge-convex}
		\min_{0\leq r\leq k}\{ \E{D}{x^r}{l} - \E{D}{x^*}{l} \} \leq& \frac{\epsilon L^c_2 \sum_{r=0}^k \eta^r\Vert x^*-x^r\Vert}{\sum_{r=0}^k\eta^r }\\
	&+ \frac{\Vert x^0-x^*\Vert^2 + (L^c_3)^2\sum_{r=0}^k (\eta^r)^2}{2\sum_{r=0}^k\eta^r}
	\end{align*}
\end{theorem}

Theorem \ref{theo::convergence-convex} is a statement about the minimum distance between generated sequence $\E{D}{x^r}{l}$ and the optimal solution, $\E{D}{x^*}{l}$ when the sample size $N$ is sufficiently large. The error bound can be divided into two parts. The first term on the right-hand side is the the bound on the decision-dependent error, which is caused by the decision-dependent characteristic (see proof of Theorem \ref{theo::convergence-convex}). We observe that the decision-dependent error depends on the distribution distance $\epsilon$ and the Lipschitz constant $L^c_2$ that indicates the change rate of cost function to the change of random parameter. In other words, if the distance between the distribution under two different solutions is small, and the cost does not sensitive to the change of random parameter, then the decision-dependent error should not be an issue. However, when the decision-dependent distribution react sensitively to the change of price decision, or the cost change rapidly to the random demand, the decision-dependent error will be large, and Theorem \ref{theo::convergence-convex} indicates that AGD algorithm may not be a good choice in such a scenario. We also note that the decision-dependent will decrease as the solution $x^r$ get close to the optimal solution $x^*$. Although we cannot guarantee that $x^r$ converge to $x^*$ in the general convex case, this term can be bounded by $D_X$, which denotes the maximum distance in the decision space (for example, the gap between maximal price and the cost), and the decision-dependent error turns into $\epsilon L^c_2D_X$ in such case.

The second term on the right-hand side is due to the internal error of the gradient descent algorithm. And the only way to decrease the internal error is iterating for more times. Since the step size $\eta^r$ is often decreasing and less than $1$, the internal error term will converge to zero as $k$ increases. 

Theorem \ref{theo::convergence-convex} can suit for both the price-only model and price-adjustment cost model, since in two cases the cost function $l(x, D)$ is concave in the decision variable $p$. This conclusion also applies to other decision dependent problems that takes the form of an expectation of a convex function. We apply the convergence result to the price-only model. Corollary \ref{coro::price-only-converge} provides an upper bound on the gap between the solution sequence and the optimal value in the price-only model after $k$ iterations. 

\begin{corollary}[Convergence for price-only model]\label{coro::price-only-converge}
	For the price-only model, we denote $\lim_{N\rightarrow \infty} p^r_N=p^r$, and $p^*$ is the optimal pricing decision. If we adopt AGD algorithm with step size $\eta^k$ for $k$ iterations,

	\begin{align*}
	\min_{0\leq r\leq k}\{f_p(p^r;z,q) - f_p(p^*;z,q) \} \leq& \frac{\epsilon \sum_{r=0}^k \eta^r|p^*-p^r|}{\sum_{r=0}^k\eta^r }\\
	&+ \frac{(p^0-p^*)^2 + (D_{m}\wedge q)^2\sum_{r=0}^k (\eta^r)^2}{2\sum_{r=0}^k\eta^r}
	\end{align*}

	where $D_m$ is the maximum possible demand.
\end{corollary}

However, it is still not clear how the solution sequence can converge to the optimal point. Therefore, we investigate the distance to the optimal solution under strongly-convex condition. 

\subsubsection*{Strongly-convex case: the price-only problem with price-adjustment cost}

Similar to \cite{DAU4-Perdomo2020}, we do not directly give the distance to the optimal point, but provide a distance bound to a stable point first. We state the distance bound between the solution sequence of AGD algorithm and the stable point $x_{PS}$ in Theorem \ref{theo::strconv-dist-stable}.

\begin{theorem}[Distance to stable points]\label{theo::strconv-dist-stable}
	
	Suppose that Assumptions \ref{assum::switch-int}, \ref{assum::eps-sensitive} and \ref{assum::conv-lip-grad} are satisfied, $l(x,D)$ is $\gamma$-strongly convex in $x$ and at least one stable point $x_{PS}$ exists. We denote $A = \gamma - \epsilon L_1^c$ and $B = L^c_1\sqrt{1+\epsilon^2}$. If $A\ge 2B\ge 0$ and we take the constant step size $\eta$ that satisfies

	\begin{equation*}
		4B^2\eta^2-2A\eta+1 = 0
	\end{equation*}

	Then for any $\xi>0$, there exists a sample size $N_0$, for all $N>N_0$, we have the following conclusion after $k+1$ iterations,

	\textit{case 1.} if $1 - 2\eta A + 2\eta^2B^2 > 0$, then

	\begin{equation}
		\Vert x^{k+1}_N-x_{PS}\Vert\leq C^{k+1}\Vert x^1_N-x^*\Vert + \xi \eta \frac{1-C^{k+1}}{1-C}
	\end{equation}

	where $C = \sqrt{1 - 2\eta A + 2\eta^2B^2} < 1$.

	\textit{case 2.} if $1 - 2\eta A + 2\eta^2B^2 \leq 0$, then

	for any $K>0$, there exists $k>K$, such that 

	\begin{equation}
		\Vert x^{k+1}_N-x_{PS}\Vert \leq (1+\sqrt 2)\xi\eta
	\end{equation}

\end{theorem}

Theorem \ref{theo::strconv-dist-stable} shows the distance two the stationary point under the strongly-convex condition. In case 1, the bound can also be divided into two components: the first term is due to the initial distance between $x^1_k$ and $x_{PS}$ and the parameter $C$. Parameter $C$ is relevant to the strongly-convex parameter $\gamma$, the Lipschitz continuous parameter $L^c_1$ and the distribution distance $\epsilon$. When the degree of strong convexity is large, and the cost function and decision-dependent distribution do not react sensitively to the decision variables, then $C$ is small and the first term decreases rapidly. This result is reasonable since strong convexity can increase the converge speed and the decision-dependent effect diminishes as $L^c_1$ and $\epsilon$ decrease. The second term can be explained by the approximate error. As shown in Proposition \ref{prop::grad-converge}, $\xi$ can be sufficiently small when the sample size is large. Therefore, we can reduce the second term by collecting more data.

In case 2, we prove that the distance will be become decreasing immediately when exceeding the bound $(1+\sqrt 2)\xi\eta$ until it reach the bound again. Therefore, we can prove that the distance will either decreasing or fluctuate around $(1+\sqrt 2)\xi\eta$.

Now we focus on the distance to the optimal solution. We have investigated the distance bound between solution sequence and stable point $x_{PS}$ in Theorem \ref{theo::strconv-dist-stable}, so we only need to show the relationship between stable point and the optimal solution. 

\begin{lemma}[Theorem 4.3 in \citep{DAU4-Perdomo2020}]\label{lem::stable-opt-distance}
	Suppose that $l(x,D)$ is $L_D$-Lipschitz in $D$ and strongly convex, and the Assumption \ref{assum::eps-sensitive} is satisfied, then for every stable point $x_{PS}$, we have 

	\begin{equation*}
		\Vert x^* - x_{PS}\Vert \leq \frac{2L_D\epsilon}{\gamma}
	\end{equation*}
\end{lemma}

Lemma \ref{lem::stable-opt-distance} shows the distance bound between the stable point and optimal solution in the strongly-convex case. Based on this conclusion, we can now give the upper bound of optimality gap of the solution. Theorem \ref{theo::strconv-dist-stable} and Lemma \ref{lem::stable-opt-distance} can be applied to the price-only model with price-adjustment cost, so we can also give the upper bound for the price-adjustment cost model.

\begin{corollary}[Distance to optimal solution]\label{coro::dist-opt}
	Suppose that all assumptions in Theorem \ref{theo::strconv-dist-stable} are satisfied, and $l(x,D)$ is $L_D$-Lipschitz in $D$, then after $k$ iterations, we have 

	\begin{equation*}
		\Vert x^{k+1} - x^* \Vert \leq \frac{2L_D\epsilon}{\gamma} + \max\left\{C^{k+1}\Vert x^1_N-x^*\Vert + \xi \eta \frac{1-C^{k+1}}{1-C}, (1+\sqrt 2)\xi\eta\right\}
	\end{equation*}
\end{corollary}

\begin{corollary}[Solution gap of the price-adjustment cost model]\label{coro::price-adj-converge}
	Denote $p^*$ as the optimal solution of the pricing problem with price-adjustment cost. Assume that $\gamma -\epsilon L^c_1 \ge 2L^c_1\sqrt{1+\epsilon^2}$, and the step size $\eta$ is taken as Theorem \ref{theo::strconv-dist-stable}. After $k$ iterations, for any $\xi>0$, there exists a sample size $N_0$, if $N>N_0$, the solution gap can be bounded by 

	\begin{equation*}
		|p^k-p^*|\leq \frac{2\epsilon}{\gamma}+\max\left\{ C^{k+1}|p^1_N-p^*|+\xi\eta\frac{1-C^{k}}{1-C}, (1+\sqrt 2 \xi\eta) \right\}
	\end{equation*}

	where $C = \sqrt{1-2\eta(\gamma - \epsilon L^c_1)+ 2\eta^2(1+\epsilon^2)(L^c_1)^2}$, $L^c_1 = 2\gamma (p-\tilde{p})$.

\end{corollary}

In summary, when the cost function is strongly-convex, we can prove the performance bound to an intermediate stable point. Since the stable point is close to the optimal solution, we can then provide the distance bound to the optimal solution of AGD algorithm under strongly-convex case. The convergence result in Theorem \ref{theo::strconv-dist-stable} is suitable for both the price-only models with and without price-adjustment cost. 

%% file: experiment.tex
\section{Experiments}\label{sec::experiment}

% 在这一节中，我们讨论AGD算法的收敛表现以及相对于其他方法的优越性。在实验中，我们使用了仿真数据和真实的电力行业定价数据验证算法的有效性和收敛表现。为了验证decision-dependent的重要性，我们将算法与不考虑decision-dependent情形的求解方法结果在both实际数据和仿真数据中进行比较。最后，为了验证distribution-free模型的现实意义，我们在仿真数据中将AGD算法与prediction-then-optimize方法比较，并给出了使用distribution-free models的marginal insights.

In this section, we validate the convergence performance of the AGD algorithm and compare its performance with other methods. In our experiments, we use both simulated data and real data from the electricity industry to validate the effectiveness the proposed AGD algorithm. To demonstrate the importance of decision-dependent characteristic, we compare the results of our algorithm with those of models that do not consider decision-dependent effect in both real and simulated data. Finally, we compare our AGD algorithm with a prediction-then-optimize approach in simulated data, which provides evidence for the necessity of using distribution-free models in practice.

\subsection{Data Description}

% We conduct the numerical experiments on two datasets. 第一个数据集来自于仿真的数据。在仿真数据集中，我们模拟了特征和需求的分布，并希望通过决策最优定价和订货量最大化实际利润。这个数据集背后的需求分布和参数值选取可以参见EC4[??]节。第二个数据集来自真实世界的电厂定价问题。该数据集描述了2015-2020年某地的用电量需求和电价情况。数据中包含的影响用电量需求的因素有温度、日照强度、假期、降雨量等。每天决策者需要根据这些特征决策当天的电量定价。决策者面临的最主要挑战在于他并不知道在当前特征下需求随定价的变化规律。他只能依靠历史定价和特征数据估计需求。需要注意的是，用电量与定价成反比这一关系在数据集中并不一定成立。在夏季和冬季，电厂往往会通过增大电价来控制用电需求，这导致在数据集中电价与需求的反比关系并不显著。这也体现了我们模型中decision-dependent性质的重要性。

We conduct numerical experiments on two datasets. The first dataset comes from a real-world power plant pricing scenario. This dataset describes the electricity demand and price situation in Victoria, Australia from 2015 to 2020. Factors affecting the electricity demand include temperature, solar exposure, holidays, rainfall, etc. The manager needs to decide on the electricity price on a daily frequency based on these features. The main challenge is that the manager does not know how demand changes with price under the current features. He/she can only estimate demand based on historical pricing and feature data. Noted that the intuitive inverse relationship between electricity consumption and price maybe not so clear in the dataset. In summer and winter when the electricity demand is high, power plants often control electricity demand by increasing prices, which results in a less significant inverse relationship between price and demand in the dataset. This also reflects the importance of decision-dependency in our model. The source of the real-world dataset can be seen in Section \ref{sec::appd-data-description}.

The second dataset comes from simulated data. We generated the demand, price and feature data from a known distribution and functional relationship. The aim is to maximize actual profits through optimal pricing and order quantity decisions. The underlying demand distribution and parameter values of this dataset can be found in Section \ref{sec::appd-data-description}.

\subsection{Convergence Performance}
\label{sec::exp-convergence-perform}

% In this section, 我们从不同方面验证了AGD算法的收敛表现。为了验证我们提出的AGD算法的收敛表现。我们首先测试了在kNN, kernel regression, decision tree (CART)和random forest (RF)下算法的收敛情况。其中CART与RF的训练通过Python中"scikit_learn"包完成。各个机器学习方法的超参数通过网格搜索的方式确定。specifically, we manually discretize the value space of hyperparameters and select the best hyperparameters in the prediction model according to the prediction performance. 最后，我们采用Armijo原则选取步长 (Armijo步长与diminishing步长的比较可参见EC4[??])。 The performance of our algorithms can be evaluated according to two criteria: (1) the convergence and (2) the change of objective value in each iteration.

In this section, we validate the convergence performance of the AGD algorithm from different aspects. We first test the convergence of the algorithm under kNN, kernel regression, decision tree (CART), and random forest (RF). The training of CART and RF is completed by the “scikit-learn” package in Python. The hyperparameters of each ML method are determined through grid search. Specifically, we manually discretize the value space of hyperparameters and select the best hyperparameters in the prediction model according to prediction performance. The iteration stops when the step size is below $10^{-5}$ or the solution exceeds the upper or lower bound. Finally, we use the Armijo principle to select the step size (for a comparison between Armijo step size and diminishing step size, see Section \ref{sec::appd-step-compare}). All computations were carried out on Python 3.10, with an Inter i7-9750H processor which has 32.0 GB of RAM.

% 我们先比较了在报童定价问题下的收敛表现。我们发现四种模型均能够收敛到局部最优点附近。其中RF和CART的表现更好，这说明RF和CART对于近似利润和近似梯度的估计更准确。如图[...]所示，我们也发现AGD算法的收敛率相对较高，能在较小的迭代次数内收敛到算法的最优点。其中，RF和CART的gap在[??]以下。

The performance of our algorithms can be evaluated according to two criteria: (1) convergence and (2) the change of objective value in each iteration. We first compare the convergence performance of the newsvendor pricing problem on the simulated dataset. According to Figure \ref{fig::converge}, we find that all four models are able to descent towards the local optimum. Among them, kernel regression and CART perform better, indicating that they have more accurate estimates of profit and gradient. Note that the priority of these two weight methods not always holds true, but depends on whether the weight method gives an accurate prediction on the model. As shown in Figure \ref{subfig::converge-gap}, we also find that the AGD algorithm has an relatively high converge rate. Among four weight functions, the gap of kNN, kernel regression and CART is below 5\%.

\begin{figure}[htbp]
    \caption{Comparison among different weight functions of AGD algorithm for newsvendor pricing model (simulated dataset).}
    \centering
	\subfigure[Process snapshot]{
		\includegraphics[width=0.45\linewidth]{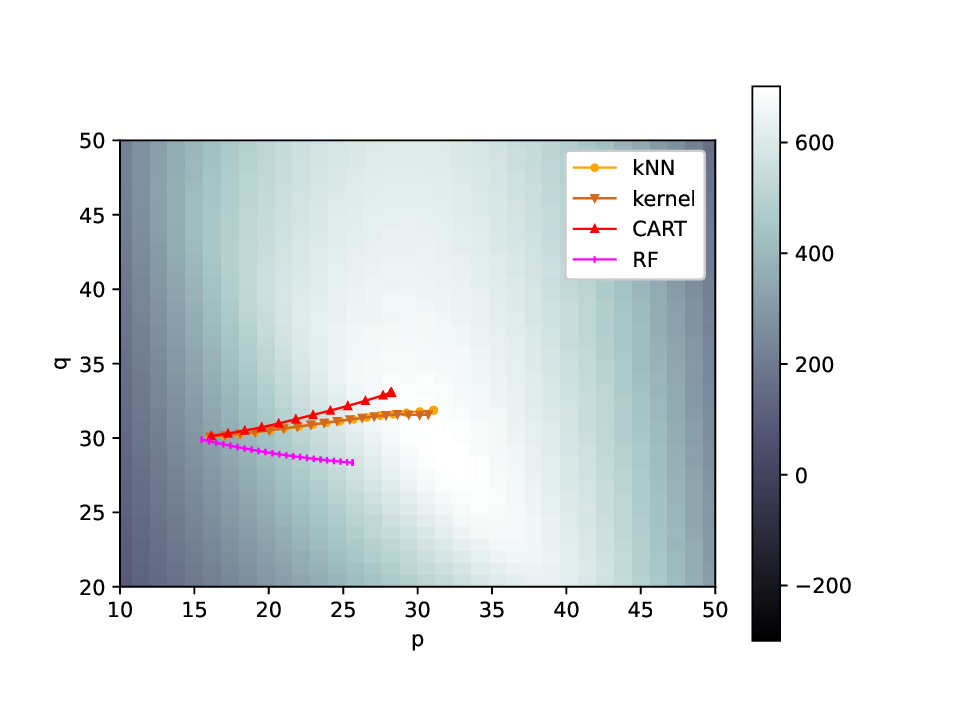}
        \label{subfig::converge3d}
    }\hfill
	\subfigure[Optimality gap]{
		\centering
		\includegraphics[width=0.45\linewidth]{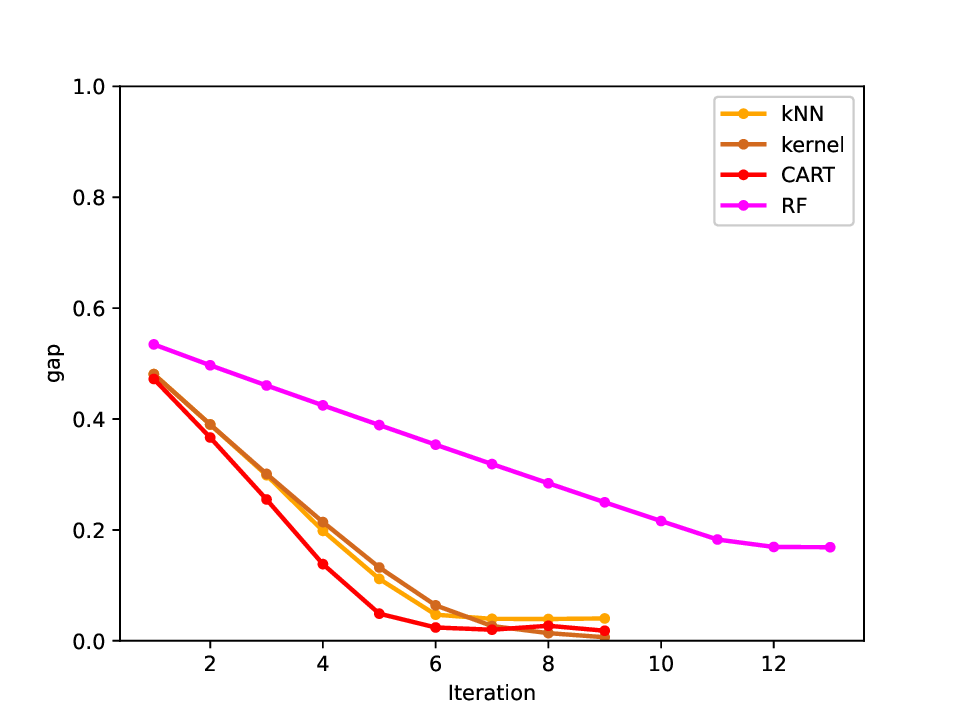}
        \label{subfig::converge-gap}
	}
    
    \label{fig::converge}
\end{figure}

% 其次，我们比较了在price-only问题和考虑adjustment-cost模型下算法的收敛性。图[??]展示了两种情况下的收敛结果。结果显示在Price-only条件下，算法离最优值的gap在[??]左右。

Secondly, we compare the convergence of the price-only model and the model considering adjustment-cost on the real-world dataset. Figure \ref{fig::price-gap} shows the convergence results in both cases. The results show that under the price-only condition, the gap between the algorithm and the optimal value is less than 5\%. We can observe that each algorithm has a slight deviation from the optimal solution before the iteration stops. This deviation reflects the approximation error of the weight approximation method. This is also why we use the Armijo rule to select the step size: the Armijo rule can ensure that the estimated function value decreases monotonically, that is, the true function value will not deviate much from the local minimum.

\begin{figure}[htbp]
    \caption{Comparison of different weight functions of AGD algorithm for price-only model with and without price-adjustment cost (real dataset). }
    \centering
	\subfigure[Price-only model]{
		\includegraphics[width=0.45\linewidth]{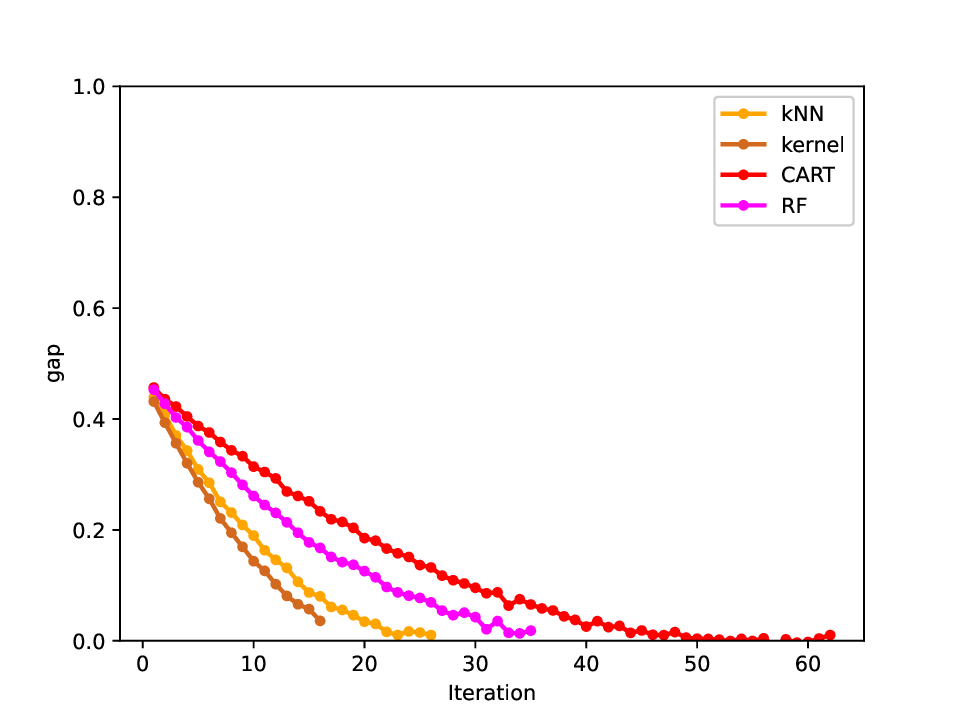}
        \label{subfig::price-only-gap}
    }\hfill
	\subfigure[Price-only model with price-adjustment cost]{
		\centering
		\includegraphics[width=0.45\linewidth]{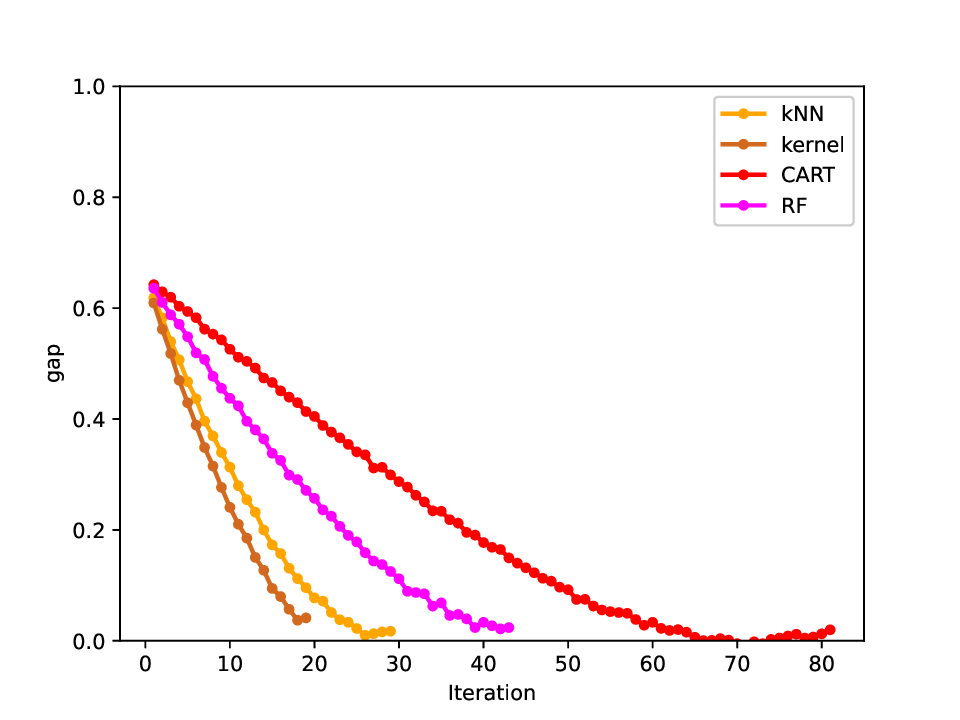}
        \label{subfig::price-only-adjust-gap}
	}
    
    \label{fig::price-gap}
\end{figure}

% 为了验证模型在少样本条件下的表现，我们研究了在不同样本量下，AGD算法在报童定价模型下的表现。如图[??]所示，我们的算法在小样本下(n<=100)也有较好的表现。

In order to verify the performance of the model under small sample conditions, we study the performance of the AGD algorithm under different sample sizes on the simulated dataset. We use kNN weight method with $k=20$ and generate $5$ stochastic features for each sample size. The average, maximum and minimum optimality gaps are shown in Figure \ref{fig::sample-size}. As expected, AGD algorithm can converges to stationary point of the true objective with small sample size. When the sample size is large, the optimality gap becomes more stable and decreases correspondingly, since a larger sample size provide more information of the true distribution.

\begin{figure}[htbp]
    \caption{Optimality gap under different sample size (simulated dataset)}
    \centering
    \includegraphics[width=0.6\linewidth]{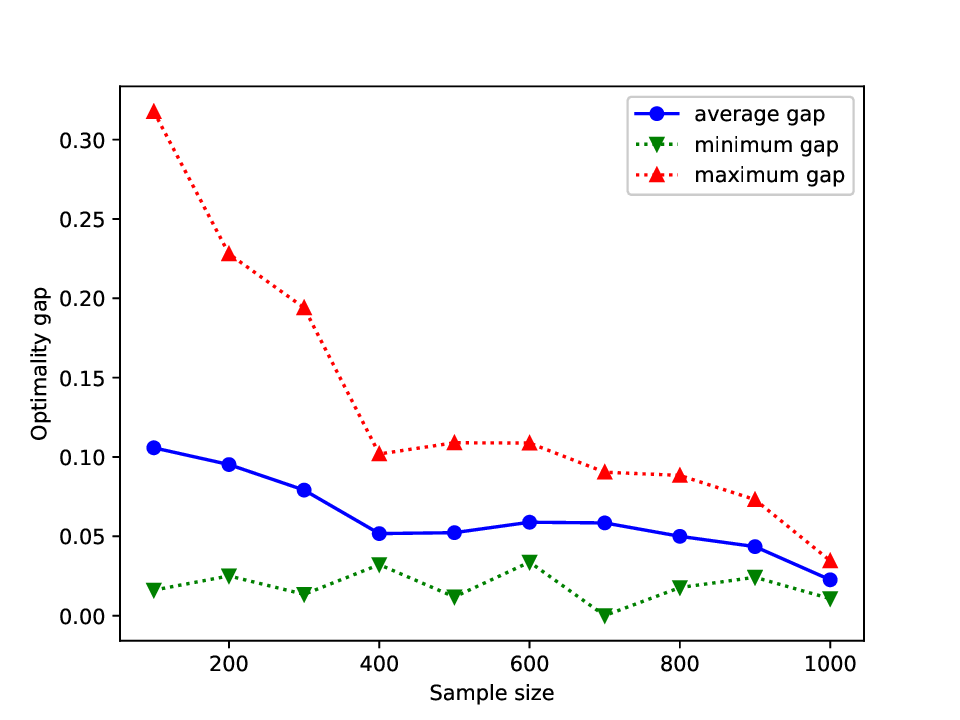}
    
    \label{fig::sample-size}
\end{figure}

% 最后，我们将AGD算法应用于真实情况下电力行业的定价问题。我们将数据集分为训练集(n=1895, 90%)和测试集(n=211, 10%)。我们根据训练集数据生成近似模型，并在测试集中抽取特征数据求解定价。我们假定测试集中的定价数据是合理的，并将算法做出的决策与实际定价进行比较来判断算法结果的合理性。图[??]的结果显示AGD算法给出的定价与真实定价的差距在[??]以内。

Finally, we apply the AGD algorithm to the pricing problem in the electricity industry under real conditions. We divided the dataset into two parts, the training set (n = 1895, 90\%) and the test set (n = 211, 10\%). We generate an approximate model based on the training set data and sample the historical features from the test set to before using them to solve the pricing problem. We assume that the pricing data in the test set is reasonable, so we can compare the decisions made by algorithm with the actual prices to judge the rationality of the algorithm results. The results in Figure \ref{fig::true-price-dist} show that the difference between the most pricing decision given by the AGD algorithm and the actual pricing is within 5\%.

\begin{figure}[htbp]
    \caption{Comparison between historical price and optimization result (real dataset)}
    \centering
    \includegraphics[width=0.6\linewidth]{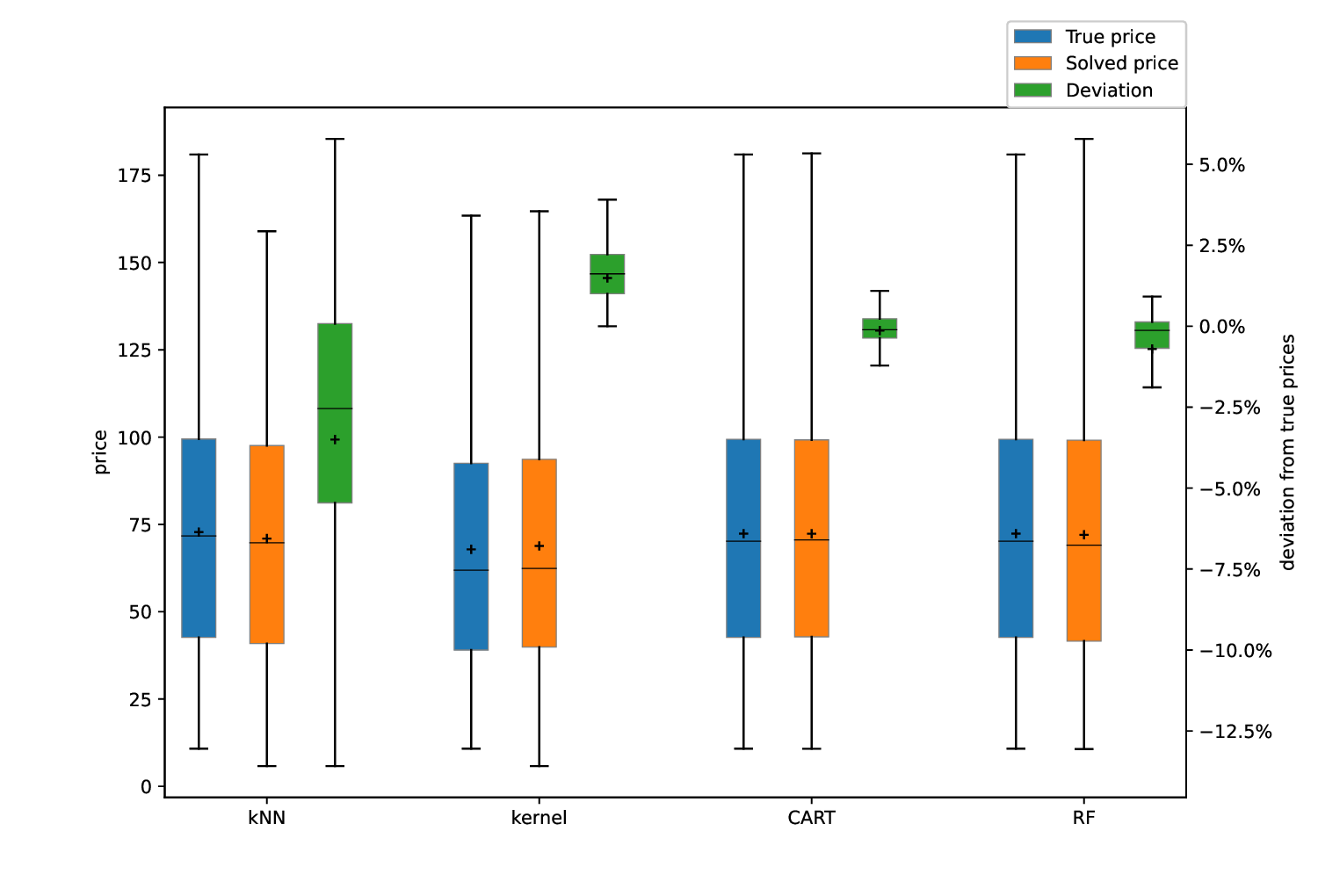}
    
    \label{fig::true-price-dist}
\end{figure}

\subsection{Value of Decision Dependent Property}

% 你现在是一个研究运筹学的学者，你在从事一项研究，你研究的模型具有 decision dependent 性质。为了求解这个模型，你设计了一个名叫AGD的算法，并进行了数值实验(numerical experiment)。现在，你需要完成论文中的一个章节。
% 这个章节的主题是【验证模型中decision-dependent性质的必要性】。你的论文应该包含以下几点：（1）我们本节的目的和主题（2）我们的实验设置(2) 在仿真数据集下的结果：在decision-dependent性质下，用AGD算法求解，结果能够收敛到局部极小值附近。而如果假设模型没有decision-dependent性质，并在该假设下建立模型求解，那么所得到的解离最优解将偏差很远。你有一个关于这个结果的条形图。(3) 在真实的电力行业定价的数据集上，如果假设问题有decision-dependent性质，得到的结果接近真实定价，而如果假设问题没有decision-dependent性质，得到的结果距离真实定价偏差很远。你有关于这一结果的箱线图。
% 现在，请你完成论文的这个章节的提纲，语言为英文(English), 请用仿照学术论文的口吻完成(academic). 不少于3个自然段。

In this section, we demonstrate the necessity of considering the decision-dependent characteristic of our model. Remind that the decision-dependent models are those in which the distribution of stochastic parameters, for example the demands, depend on the decision itself, rather than being determined solely by the features. Failure to account for this property can lead to significant errors in the solution obtained. We will compare two solution approaches: the decision-dependent scenario solved by AGD and the decision-independent scenario solved by reliable optimization tools. We show that accounting for the decision-dependent nature of the problem is crucial for obtaining accurate solutions.

We first describe the two scenarios mentioned above. In the decision-dependent scenario, we construct approximate model as (\ref{eq::appr-obj}) and use AGD algorithm to solve the model. In the decision-independent scenario, however, we construct the approximate model in the same way as the decision-dependent scenario. The only difference is that we delete the price variable $p$ from the weight function $w^i(\cdot)$. We use the minimization function in "scipy" package in Python, since the removal of decision variable from the weight function significantly reduces the complexity of the approximate model. We also study the performance of pure SAA method, which ignores both decision-dependent effect and feature (see Section \ref{sec::appd-pureSAA} for detail).

We conduct the comparison on both simulated and real-world datasets. In the simulated dataset experiments, we compare the objective function obtained from the two scenarios. In the real-world dataset, since we are unable to determine the true objective function due to the unknown demand distribution, we instead assume that the historical pricing decisions in the test set are reasonable. We then compare the optimization results obtained from decision-dependent, decision-independent and pure SAA scenarios to the historical prices in the test set. We set the deviation from the historical prices as a criterion of the performance of the algorithm.

The result on the simulated dataset shown in Table \ref{tab::decision-dep}, where the gap reduction is the difference between decision-independent approach and decision-dependent approach. When the decision-dependent characteristic of the problem is taken into account, the AGD algorithm can converge to a local minimum in the vicinity of the true optimum. However, when the decision-dependent is not taken into consideration, the obtained solution is significantly biased away from the true optimum. 

% Table generated by Excel2LaTeX from sheet 'n=1000'
% \begin{table}[htbp]
%     \centering
%     \caption{Comparison between decision-dependent and decision-independent scenarios (real dataset). }
%       \begin{tabular}{cccccc}
%       \toprule
%       Method & \multicolumn{1}{l}{kNN} & \multicolumn{1}{l}{kernel} & \multicolumn{1}{l}{CART} & \multicolumn{1}{l}{RF} &  \\
%       \midrule
%       Profit & 674.37 & 698.20 & 689.94 & 584.07 &  \\
%       Optimality gap & 4.03\% & 0.64\% & 1.82\% & 16.88\% &  \\
%       Gap reduction & 96.93\% & 99.51\% & 98.49\% & 86.41\% &  \\
%       \midrule
%       Method & \multicolumn{1}{l}{kNN-indep} & \multicolumn{1}{l}{kernel-indep} & \multicolumn{1}{l}{CART-indep} & \multicolumn{1}{l}{RF\_indep} & \multicolumn{1}{l}{pure SAA} \\
%       \midrule
%       Profit & -220.88 & -212.52 & -139.59 & -170.46 & -283.62 \\
%       Optimality gap & 131.43\% & 130.24\% & 119.87\% & 124.26\% & 140.36\% \\
%       \bottomrule
%       \end{tabular}%
%     \label{tab::decision-dep}%
% \end{table}%

% Table generated by Excel2LaTeX from sheet 'decision-dep比较'
\begin{table}[htbp]
    \centering
    \caption{Comparison between decision-dependent and decision-independent scenarios (real dataset). }
      \begin{tabular}{cccccc}
      \toprule
      \multirow{2}[4]{*}{Method} & \multicolumn{4}{c}{Decision-dependent scenarios} &  \\
  \cmidrule{2-5}          & kNN   & kernel & CART  & RF    &  \\
      \midrule
      Profit & 674.37 & 698.2 & 689.94 & 584.07 &  \\
      Optimality gap & 4.03\% & 0.64\% & 1.82\% & 16.88\% &  \\
      Gap reduction & 96.93\% & 99.51\% & 98.49\% & 86.41\% &  \\
      \midrule
      \multirow{2}[4]{*}{Method} & \multicolumn{4}{c}{Decision-independent scenarios} &  \\
  \cmidrule{2-5}          & kNN   & kernel & CART  & RF    & pure SAA \\
      \midrule
      Profit & -220.88 & -212.52 & -139.59 & -170.46 & -283.62 \\
      Optimality gap & 131.43\% & 130.24\% & 119.87\% & 124.26\% & 140.36\% \\
      \bottomrule
      \end{tabular}%
    \label{tab::decision-dep}%
\end{table}%

The results of the real-world pricing problem can be seen in Figure \ref{fig::decision-dep-price-dist}. When the decision-dependent characteristic of the problem is considered, the obtained pricing decision is close to the true historical price, while the solution obtained under the decision-independent assumption was significantly biased away from the historical prices, which is consistent with the simulated dataset.

\begin{figure}[htbp]
    \caption{Deviation comparison between decision dependent and independent scenarios (real dataset)}
    \centering
    \includegraphics[width=0.6\linewidth]{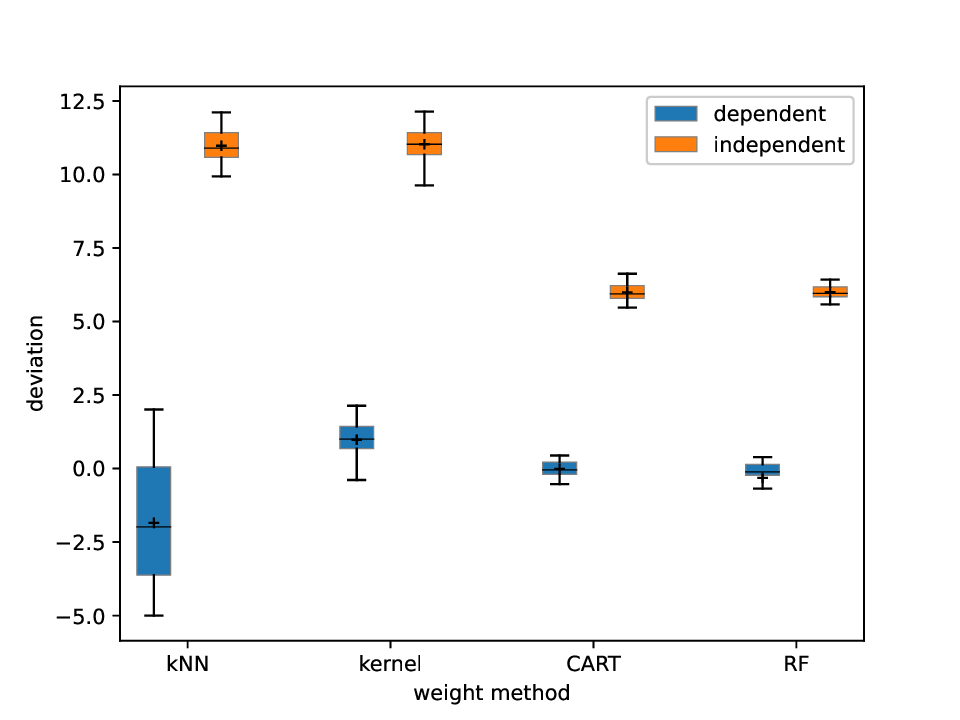}
    
    \label{fig::decision-dep-price-dist}
\end{figure}

In summary, we demonstrate the importance of accounting for the decision-dependent characteristic of our weight approximation model. In both simulated and practical dataset, failure to consider the decision-dependent characteristic can result in significant errors in the solution obtained. This proves the necessity of decision-dependent characteristic and hence illustrates the importance of AGD algorithm that can solve this kind of problem.

\subsection{Value of Distribution-Free Property}

% 现在，你还是一个研究运筹学的学者，你研究的模型具有 distribution-free性质, 这个性质的含义是模型中随即参数的分布未知，只能通过历史数据来估计。同样，为了求解这个模型，你设计了一个名叫AGD的算法，并进行了数值实验(numerical experiment)。现在，你需要完成论文中数值实验部分的一个章节。
% 这一部分的主题是：【验证distribution-free性质的必要性和实际意义】。你的论文章节应该包含以下几点: 
% 1. 实验设置和实验步骤: 你用两个分布分别生成了两组需求。第一组需求和特征、价格呈简单的线性关系，且误差项服从正态分布。即第一组需求样本正好满足线性回归的假设。第二组需求数据来源于较复杂的需求模型(参见Lin et al, 2022)。对每组需求样本，我们采用两种求解方法：第一种采用Predict-then-optimize模式，现根据样本用线性回归构建需求模型D(p,z)，然后带入目标函数中求解用最优化方法求解。第二种则采用section[??]中的distribution-free近似模型，然后采用AGD算法进行求解。
% 2. 实验结果: 实验的结果如图[??]所示。在简单的线性需求模型下，第一种Predict-then-optimize的方法优于distribution-free方法，这是因为需求模型正好符合需求预测的假设。然而即使在这种情况下, distribution-free模型的表现也不差。而在复杂分布的情形下，显然distribution-free的模型表现更好。我们发现当prediction的假设偏离真实分布的规律时，predict-then-optimize策略的利润普遍低于distribution-free模型的利润。而且前者有时还会出现不收敛的问题。因此，当缺乏对随即参数分布的信息时，相比于对未知参数的分布规则进行事先假定，采用distribution-free方法能够适应更多现实场景下的参数分布，并获得更好的解。

In this section, we validate the necessity and practical significance of the distribution-free property in our model. The experiments are designed to compare the performance of the distribution-free (DF) method with the traditional PTO approach in two scenarios with different demand distributions.

We generate two sets of demand samples, one with the simple linear decision rule where the demands have a linear relationship to the features and price, and the other with the complex demand model we used to generate simulated dataset (see Section \ref{sec::appd-data-description}). For each demand setting, we compare the performance of two solution approaches: the first approach is to use the PTO paradigm. We first construct the demand model $D(p,z)$ by linear regression trained by the dataset before substituting the demand model and optimizing the best price. The second approach is to use the distribution-free approximation model described in Section \ref{sec::model}, and then solve the problem using the AGD algorithm.

The results of our experiments are shown in Figure \ref{fig::distribution-free}. We can observe that in the simple linear demand model scenario, the PTO method outperforms the distribution-free method because the true demand model satisfies the demand prediction assumptions exactly. We also notice that even in this case, the distribution-free method still performs well. In contrast, in the complex demand distribution scenario, our DF method clearly outperforms the PTO method. Moreover, we find that when the demand prediction assumption deviates from the true distribution, the PTO strategy sometimes generates lower profits and sometimes fails to converge. Therefore, when there is few information about the distribution of stochastic parameters, adopting the distribution-free method can better adapt to a broader range of real-world scenarios, and lead to better solutions than assuming a specific distribution rule for unknown parameters.

\begin{figure}[htbp]
    \caption{Comparison between distribution-free (DF) and PTO approach under two demand models (simulated dataset)}
    \centering
	\subfigure[Profits comparison between DF and PTO method under two demand models]{
		\includegraphics[width=0.45\linewidth]{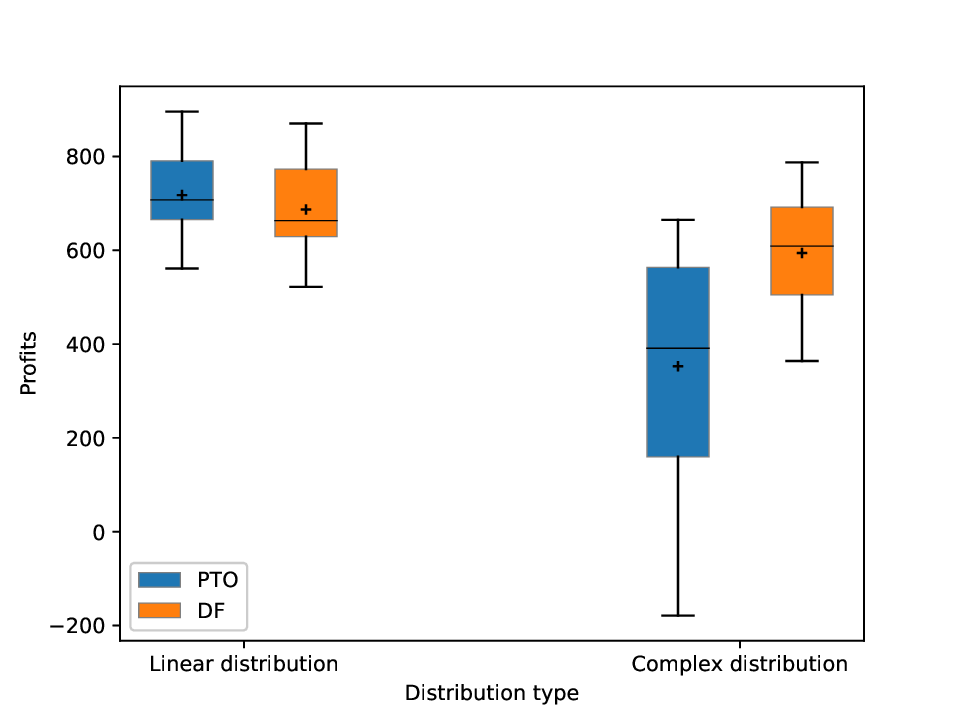}
        \label{subfig::DF-value}
    }\hfill
	\subfigure[Profit differences of DF and PTO method under two demand models]{
		\centering
		\includegraphics[width=0.45\linewidth]{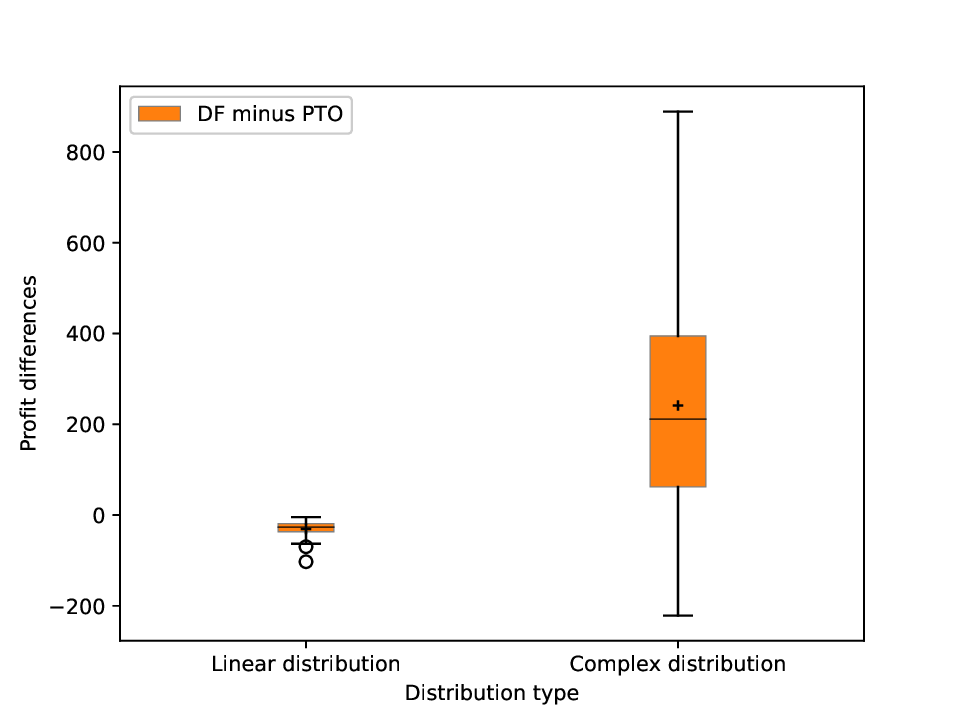}
        \label{subfig::DF-diff}
	}
    
    \label{fig::distribution-free}
\end{figure}

In summary, our experiments show that the distribution-free property is necessary and beneficial for our model. Though increases the complexity, it can help to overcome the challenge of estimating the unknown demand distribution, and is more robust to various distributional assumptions that may not hold in practice. The result also suggests that the AGD algorithm can handle the complexity brought by the distribution-free property.

%% file: conclusion.tex
\section{Conclusion}\label{sec::conclusion}

In this work, we study a data-driven, decision-dependent newsvendor pricing problem. The demand distribution is unknown and depends on pricing decision. The demand has to be learned from the historical features and prices. We develop the approximation model to address the distribution-free characteristic of demand. But we find the decision-dependent characteristic brings significant complexity to the approximate model. Therefore, we develop the concept of approximate gradient and propose the AGD algorithm to efficiently and effectively solve the approximate model. We show that our the theoretical guarantee of our AGD algorithm under nonconvex, convex and strongly-convex conditions. Our simulation experiments demonstrate that our model can approximately converge to the local optimal point, even with a small sample size Moreover, the results on the real-world dataset indicates that the pricing decisions generated by the AGD algorithm are similar to those made in practice. 

% 我们在仿真数据集上的研究显示我们的模型能够逐步逼近局部最优点，而且这在较小的样本数量下也能实现。在仿真数据集下，用AGD算法生成的定价也与现实中定价决策相似。我们发现了相比于传统报童模型，报童定价问题特有的decision-dependent性质在现实问题中非常重要。在一些行业需求并不仅仅取决于价格，但是如果忽略价格决策对需求的影响，会导致决策在很大程度上偏离最优值。此外，我们也对数据驱动方法的优势提出了insights。虽然当需求分布恰好和预测模型十分接近时，采用predict-then-optimize的方法更好。但在现实问题中我们无法确保预测模型的准确性，尤其是在新产品或新定价初期。这时采用AGD algorithm的数据驱动方法可以适应复杂且未知的需求分布。

Our study finds that compare to the traditional newsvendor problem, the decision-dependent newsvendor pricing problem is important. Though the demand may not be solely determined by price, ignoring the impact of pricing decisions on demand can result in decisions that deviate significantly from the optimal profit. Furthermore, we highlight the advantages of data-driven methods, and provide insights into the priority of distribution-free methods. When the demand distribution is close to the predictive model, the PTO framework is preferred. But the predict model may fail when predicting the demand of new products or initial pricing. In these cases, the AGD algorithm-based data-driven approach can adapt to complex and unknown demand distributions and decision-dependent rule.

% 我们的工作并不局限于报童模型的研究。如[??]所示，无论是approximate model的构建还是AGD算法的实现，都不依赖利润函数的具体形式。因此，我们的工作可以拓展到所有具有decision-dependent性质的模型中。Moreover, 我们所提出的Approximate gradient的思想也并不依赖具体的近似形式。我们之所以采用weight approximation方法，是因为这种方法能有效适应decision-dependent性质。如果有更好的近似方式，这一算法将达到更高的准确度。

Our work is not limited to the study of inventory management problems. As shown in (\ref{eq::appr-obj}) and (\ref{eq::appr-grad}), neither the construction of approximate models nor the implementation of the AGD algorithm is dependent on the form of the profit function. Therefore, our work can be extended to most unconstrained decision-dependent models such as the strategic classification problem in ML. Moreover, the concept of approximate gradient is not limited by the specific approximation approach. In our work, we use the weight approximation method because it is effective in handling the decision-dependent property. The AGD algorithm can also adapt to other approximation approach such as SAA. It will achieve high accuracy if the corresponding approximate approach is reliable.

% 在我们未来的研究中，我们考虑将数据驱动方法应用在多周期库存模型中。在多期模型中，需求将不仅仅取决于当期价格，还可能取决于历史价格决策。而price-adjustment cost的概念也将更具有现实意义。为了考虑多期利润的最优值，我们可能会将数据驱动的近似方法与动态规划方法结合。将这类多期库存问题转化为强化学习的定价问题，并结合数据驱动的方法也是一个可能的方向。

In our future work, we seek to apply the our data-driven method to multi-period inventory models. In multi-period models, the demand will depend not only on the current price but also on historical pricing decisions. The concept of price-adjustment cost will also be more important. To consider the optimal value of multi-period profit, we may combine data-driven approximation methods with dynamic programming methods. Transforming this type of multi-period inventory problem into a reinforcement learning problem and combining it with data-driven methods may also be a possible direction for future research. We are also interested in applying the concept of approximate gradient to constrained problems, since the gradient of constraint functions can also be approximated through the same approach, which can connect the approximate model to some convergence conclusions.

%% file: appendix.tex
\section{Definition of Weight Functions}\label{sec::appd-weight-def}

In this section, we present some definitions of the weight functions that can be used to construct the approximate model (\ref{eq::appr-obj}). 

\begin{definition}[Kernel regression weight]
	We can use the kernel function that measures the distances in $(p, z)$ to construct the weight function: 

	\begin{equation}
		\label{eq::weight-kernel}
		w^{\mbox{KR}, i}(p,z) = \frac{K_h\left((p, z) - (p^i,z^i)\right)}{\sum_{j=1}^nK_h((p, z) - (p^j,z^j))}
	\end{equation}
	
	where $K_h: \mathbb{R}^{dim(z)+1}\rightarrow \mathbb{R}$ is the kernel function with bandwidth $h$. Common kernel functions include the uniform kernel, triangular kernel and Gaussian kernel. If not noted, the kernel functions below refer to the Guassian kernel function: 
	
	\begin{equation}
		K(z) = \frac{1}{\sqrt{2\pi}}\mbox{exp}^{-\Vert z\Vert_2^2/2}
	\end{equation}
\end{definition}

\begin{definition}[CART weight]
	The CART weight functions are given by: 

	\begin{equation}
		\label{eq::weight-CART}
		w^{\mbox{CART}, i}(p, z)=\frac{\mathbb{I}\{R((p,z))=R((p^i, z^i))\}}{|\{j:R((p^j,z^j))=R((p,z))\}|}
	\end{equation}
	
	where $R: Z \rightarrow \{1,...,r\}$ is the function that maps features to the $r$ leaves on the CART. In the CART, a leaf is a collection of sample points that are classified to the same group. 
\end{definition}

\begin{definition}[Random forest weight]
	The random forest weight functions are given by: 

	\begin{equation}
		\label{eq::weight-RF}
		w^{\mbox{RF}, i}(p, z)=\frac{1}{N_E}\sum^{N_E}_{e=1}w^{\mbox{CART}, i, e}(p, z)
	\end{equation}
	
	where $N_E$ is the number of estimators in the random forest, and $w^{\mbox{CART}, i, e}(p, z)$ is the CART weight of the $e$th estimator in the random forest. 
	
	One of the advantage of random forest weight is that the variance will not get large as $N_E$ increases, while the estimation will be more accurate. The only cost is that it will consume more time to calculate the random forest weight if $N_E$ become larger.
\end{definition}

\section{Description of Solution Methods}

% 介绍数值实验中各个方式的基本解法

In this section, we explain the solution methods adopted in the numerical experiment section. 

\subsection{Diminishing Step}

The diminishing step adopt the step size $\eta^r$ such that $\eta^r>\eta^{r+1}$ and $\sum_{r=0}^{\infty}\eta^r=\infty$. A typical choice is $\eta^r=C/(r+1)$, where $C$ is a constant that can be adjusted to suit different problems. 

\subsection{Armijo Step}

Let $f(\cdot)$ denote the objective function we want to minimize. The Armijo principle chooses the step size $\eta^r$ by the following steps (we denote the approximate gradient as $G^N$ and the ascent direction as $d^r$) in algorithm \ref{algo::armijo}

\begin{algorithm}[ht]
	\renewcommand{\algorithmicrequire}{\textbf{Input:}}
	\renewcommand{\algorithmicensure}{\textbf{Output:}}
	\caption{Armijo step size}
	\label{algo::armijo}
	\begin{algorithmic}[1]
		\Require iteration solution $x^r$, $\alpha_0$, $\beta \in (0,1)$, $\sigma \in [0, 1)$, tolerance $\epsilon$.
		\Ensure step size $\eta^r$;
		\State $\eta^r = \alpha_0$;
        \State $x^{r+1} = x^r + \eta^r d^r$. 
		\While {$\eta^r \ge \epsilon$ and $f(x^{r}) - f(x^{r+1})< \sigma\eta^r (G^N)^Td^r$}
			\State $\eta^r = \eta^r * \beta$;
			\State $x^{r+1} = x^r + \eta^r d^r$. 
		\EndWhile
		\If {$\eta^r < \epsilon$}
            \State \Return Stop the iteration
        \Else
            \State \Return $\eta^r$
        \EndIf
	\end{algorithmic}
\end{algorithm}

Note that the hyperparameter $\sigma$ can be $0$ in our problem. When $\sigma=0$, the armijo step size ensure that the objective function descent in an approximate context. We also show the special meaning when $\sigma=0$ in Proposition \ref{prop::converge-stationary-armijo}. 

\subsection{Decision-independent approach}

The decision independent approach adopt a different model than the AGD algorithm: 

\begin{align}
	\mbox{(\textbf{Decision-independent Model})}\min_{p,q}\hat f(p,q;z)\overset{def}{=}\sum_{i=1}^Nw^i(z)l_a(p,q,D^i), \label{eq::appr-obj-decision-indep}
\end{align}

The only difference between (\ref{eq::appr-obj-decision-indep}) and the original approximate model (\ref{eq::appr-obj}) is the variables in the weight function. In the decision-independent model, we remove the decision variable from the weight function as an ignorance of the decision-dependent effect. We then use some optimization tool to solve (\ref{eq::appr-obj-decision-indep}) since it is much easier to solve without the appearance of decision variables in the weight function. We can also use AGD algorithm to solve (\ref{eq::appr-obj-decision-indep}) but the solution sequence is not convergent since the estimation on the gradient become worse without the consideration on the decision-dependent effect.

\subsection{Pure SAA}\label{sec::appd-pureSAA}

The Pure SAA approach is an approximate model that ignores features: 

\begin{align}
	\mbox{(\textbf{Pure SAA Model})}\min_{p,q}\hat f(p,q;z)\overset{def}{=}\frac 1 n \sum_{i=1}^Nl_a(p,q,D^i), \label{eq::appr-obj-decision-pureSAA}
\end{align}

The pure SAA model ignores both the decision-dependent effect and the feature. It can be treated as a decision-independent model without feature or an approximate model with kNN $k=N$ weight function.

\subsection{Predict-then-optimize}

The predict-then-optimize approach first predict the demand by applying linear regression on the training set. After performing the linear regression, we get the demand function $D(p, z)$. Then we can substitute the demand in the cost function $l(p,q,D)$ with $D(p,z)$ and optimize the model on $p, q$. Note that the prediction method is not limited to linear regression, all predict approaches that can lead to a demand function $D(p,z)$ can fit the PTO paradigm.

\section{Proofs}

\begin{proof}{Proof of Proposition \ref{prop::grad-converge}}
	Proposition \ref{prop::grad-converge} is actually a corollary of theorem EC.9 in \cite{Bertsimas-2019}. To prove the proposition, we only need to validate the assumptions.
	
	Since for every $p$, the marginal distribution of $D,Q\sim f_{D}(p,z)$ is independent of $p$ conditioned on $z$, the ignorability assumption satisfies. The cost gradient defined in (\ref{eq::profit-grad}) is well defined since the historical demand and supply cannot go to infinity. And it is also equicontinuous. The feasible region for $p,q$ is nonempty, closed and bounded we only restrict the up and down limit of the two decisions. Therefore, Proposition \ref{prop::grad-converge} follows by EC.9 in \cite{Bertsimas-2019}.
\end{proof}

\begin{proof}{Proof of Proposition \ref{prop::fail-converge-grad}}
	Assume that for any $p,q,z$, the value range of $D$ remains to be the same set $\Omega$. We rewrite the objective expectation to the integrate form: 

	\begin{equation*}
		\partial_{p,q}\mathbb{E}_{D\sim f_{D}(p,z)}[l(p,q,D)] = \partial{p,q}\int_{D\in\Omega} l(p,q,D) f_{D}(p,q,D) dD
	\end{equation*}

	Suppose that the derivative of $l(p,q,D) f_{D}(p,q,D)$ can be bounded by an $L^1$ function $g$ for all $p,q,D,Q$, then the derivative and integration operator can be switched. 

	\begin{equation*}
	\begin{aligned}
		\partial_{p,q}\mathbb{E}_{D\sim f_{D}(p,z)}[l(p,q,D)] =& \int_{D\in\Omega}\partial_{p,q} \left( l(p,q,D) f_{D}(p,q,D) \right) dD \\
		=& \int_{D\in\Omega} \left(\partial_{p,q} l(p,q,D)\right) f_{D}(p,q,D) dD \\
		&+ \int_{D\in\Omega} \left(\partial_{p,q} f_{D}(p,q,D)\right) l(p,q,D) dD \\
		=& \mathbb{E}_{D\sim f_{D}(p,z)}[\partial_{p, q}l(p,q,D)] \\
		&+ \int_{D\in\Omega} \left(\partial_{p,q} f_{D}(p,q,D)\right) l(p,q,D) dD
	\end{aligned}
	\end{equation*}

	Generally, the second term of the last equation is not $0$. So the expectation of cost gradient do not equal to the gradient of objective expectation and thus the convergence of approximate gradient fails.
\end{proof}

Before we begin to proof the convergence results, we first state some important results. The following lemmas show how Assumption \ref{assum::eps-sensitive} affects the distance between expectations of different distributions.

\begin{lemma}{Kantorovich-Rubinstein}\label{lem::Kanto-Rubin}
    For all function $f$ that is $1-$Lipschitz

    \begin{equation*}
        \Vert\mathbb{E}_{d\sim D(p)}\mathbb{E}[f(d)] - \mathbb{E}_{d\sim D(p')}\mathbb{E}[f(d)]\Vert \leq W_1(D(p), D(p'))
    \end{equation*}

\end{lemma}

\begin{lemma}
    \label{lem::eps-distribution-distance}
	Suppose Assumption \ref{assum::eps-sensitive} holds. Let $f: R^n\rightarrow R^d$ be an $L-$Lipschitz function, and let $X, X'\in \mathbb{R}^n$ be random variables such that $W_1(X, X')\leq C$. Then

	\begin{equation}
		\Vert \mathbb{E}[f(X)] - \mathbb{E}[f(X')] \Vert_2\leq LC
	\end{equation}
\end{lemma}

\begin{proof}{Proof of Lemma \ref{lem::eps-distribution-distance}}

    Since

    \begin{equation*}
        \begin{aligned}
        \Vert \mathbb{E}[f(X)] - \mathbb{E}[f(X')] \Vert_2^2 &= (\mathbb{E}[f(X)] - \mathbb{E}[f(X')])^T(\mathbb{E}[f(X)] - \mathbb{E}[f(X')]) \\
        &= \Vert \mathbb{E}[f(X)] - \mathbb{E}[f(X')] \Vert_2\frac{(\mathbb{E}[f(X)] - \mathbb{E}[f(X')])^T}{\Vert \mathbb{E}[f(X)] - \mathbb{E}[f(X')] \Vert_2}(\mathbb{E}[f(X)] - \mathbb{E}[f(X')])
        \end{aligned}
    \end{equation*}

    we define the unit vector $e:=\frac{(\mathbb{E}[f(X)] - \mathbb{E}[f(X')])^T}{\Vert \mathbb{E}[f(X)] - \mathbb{E}[f(X')] \Vert_2}$, we can get: 

    \begin{equation*}
        \Vert \mathbb{E}[f(X)] - \mathbb{E}[f(X')] \Vert_2^2=\Vert \mathbb{E}[f(X)] - \mathbb{E}[f(X')] \Vert_2(\mathbb{E}[e^Tf(X)]-\mathbb{E}[f(X')])
    \end{equation*}

    Since $f$ is a one-dimensional $L-$lipschitz function, we can apply Lemma \ref{lem::Kanto-Rubin} and Assumption \ref{assum::eps-sensitive} to obtain that for all $e$,

    \begin{equation*}
        \Vert \mathbb{E}[f(X)] - \mathbb{E}[f(X')] \Vert_2^2\leq \Vert \mathbb{E}[f(X)] - \mathbb{E}[f(X')] \Vert_2 LC
    \end{equation*}

    Thus completing the proof
\end{proof}

\begin{proof}{Proof of Proposition \ref{prop::converge-stationary}}

    The proof is divided into two steps. In the first step, we prove that the objective function $\E{D}{x}{l}$ has Lipschitz gradient in $x$. Then we prove that under diminishing step, any converging subsequence converge to the stationary point.
    
    We denote $f(x) = \E{D}{x}{l}$, then 
    
    \begin{equation*}
        \Vert \nabla_xf(x) - \nabla_xf(y)\Vert = \Vert\nablaE{D}{x}{l}-\nablaE{D}{y}{l}\Vert
    \end{equation*}
    
    According to Assumption \ref{assum::switch-int}, we can write the expectation to integrate form and change the integrate operator and derivative operator. 
    
    \begin{equation*}
    \begin{aligned}
        \Vert \nabla_xf(x) - \nabla_xf(y)\Vert=&\Vert\int_{D\in\Omega }\nabla_x (l(x,D)f_D(x,D))dv - \int_{D\in\Omega} \nabla_x (l(y,D)f_D(y,D))dv\Vert\\
        \leq& \Vert \int_D l(x,D)(\nabla_xf_D(x,D))-l(y,D)(\nabla_xf_D(y,D))dv \Vert \\ 
        &+ \Vert \int_D(\nabla_xl(x,D)) f_D(x,D) - (\nabla_x l(y,D))f_D(y,D)dv \Vert \\ 
        =& I + II
    \end{aligned}
    \end{equation*}
    
    The second inequality follows by the multiplication rule of derivative. We then analyze $I$ and $II$ respectively.
    
    \begin{equation*}
    \begin{aligned}
        I\leq& \Vert\int_Dl(x,D)\nabla_xf_D(x,D)dv - \int_Dl(x,D)\nabla_xf_D(y,D)dv\Vert\\
        &+\Vert\int_Dl(x,D)\nabla_xf_D(y,D)dv - \int_Dl(y,D)\nabla_xf_D(y,D)dv\Vert\\
        \leq& \int_D|l(x,D)|\Vert \nabla_xf_D(x,D) - \nabla_xf_D(y,D) \Vert dv \\
        &+ \int_D |l(x,D)-l(y,D)|\Vert \nabla_xf_D(y,D) \Vert dv \\
        \leq& S_\Omega L_4L_3\Vert x-y\Vert + S_\Omega L_5L_1\Vert x-y\Vert
    \end{aligned}
    \end{equation*}
    
    The first inequality holds from the triangular inequality. The second inequality holds by the Cauchy-Schwarz inequality. The third inequality holds by the Lipschitz continuous characteristic and intermediate value theorem, where $S_\Omega$ denotes of the set $\Omega$.
    
    We can also bound the second term by the following steps: 
    
    \begin{equation*}
    \begin{aligned}
        II\leq& \Vert \int_D \nabla_x l(x,D) f_D(x,D)dv - \int_D\nabla_xl(x,D)f_D(y,D)dv \Vert\\
        &+ \Vert\int_D\nabla_xl(x,D)f_D(y,D)dv - \int_D\nabla_xl(y,D)f_D(y,D)dv\Vert\\
        =& \Vert \Enabla{D}{x}{l} - \mathbb{E}_{D\sim f_D(y,D)}\nabla_x l(x,D)\Vert + \Vert \mathbb{E}_{D\sim f_D(y,D)}[\nabla_xl(x,D)-\nabla_xl(x,D)] \Vert \\
        \leq& \epsilon L_2\Vert x-y\Vert + L_2\Vert x - y\Vert
    \end{aligned}
    \end{equation*}
    
    The first inequality holds by the triangular inequality. The first equality holds by the definition of expectation. The second inequality holds by Lemma \ref{lem::eps-distribution-distance} and the definition of Lipschitz gradient.
    
    Thus, $\Vert \nabla_xf(x) - \nabla_xf(y)\Vert \leq [(\epsilon+1)L_2 + S_\Omega(L_1L_5+L_4L_3)] \Vert x-y\Vert$. Hence the objective function has Lipschitz gradient and $L = (\epsilon+1)L_2 + S_\Omega(L_1L_5+L_4L_3)$.
    
    recall that the update rule is given by
    
    \begin{equation*}
        x^{r+1}_N = x^r_N + \eta^rG^N(x;z) 
    \end{equation*}
    
    From descent lemma, we have 
    
    \begin{equation*}
        f(x^{r+1}_N)\leq f(x^r_N) + \eta^rG^N(x^r_N;z) ^T \nabla f(x^r_N) + \frac{L(\eta^r)^2}{2} \Vert G^N(x^r_N;z)\Vert^2
    \end{equation*}
    
    Taking $N\rightarrow \infty$ on both sides, since $f(x)$ is continuous and $\lim_{N\rightarrow \infty}G^N(x;z) = \Enabla{D}{x}{l}$ from Proposition \ref{prop::grad-converge}. 
    
    \begin{equation*}
            f(x^{r+1})- f(x^r) \leq \eta^r\Enabla{D}{x^r}{l} ^T \nabla f(x^r) + \frac{L(\eta^r)^2}{2} \Vert \Enabla{D}{x^r}{l}\Vert^2
    \end{equation*}
    
    where $x^r = \lim_{N\rightarrow\infty} x^r_N$. 
    
    Since
    
    \begin{equation*}
    \begin{aligned}
        \Enabla{D}{x^r}{l}^Tf(x^r) =& \Vert \nablaE{D}{x^r}{l} \Vert^2 + \Vert\Enabla{D}{x^r}{l} \Vert^2 \\
        &-\Vert \int_Dl(x^r,D)\nabla_x f_D(x^r,D)dv \Vert^2 \\
        \ge& (\Vert \nablaE{D}{x^r}{l} \Vert^2 - L_4^2L_5^2S^2_\Omega) + \Vert\Enabla{D}{x^r}{l} \Vert^2
    \end{aligned}
    \end{equation*}
    
    Note that since the range of $D$ is limited, we can scale the random parameters $D$ so that $S_\Omega$ is sufficiently small. Thus
    
    \begin{equation*}
        \Enabla{D}{x^r}{l}^Tf(x^r) \ge \Vert\Enabla{D}{x^r}{l} \Vert^2 
    \end{equation*}
    
    Therefore, 
    
    \begin{equation*}
        f(x^{r+1})- f(x^r)\leq -\eta^r (1 - \frac{L\eta^r}{2}) \Vert\Enabla{D}{x^r}{l} \Vert^2 
    \end{equation*}
    
    Since $\eta^r$ is diminishing, for any $\xi\in (0,1)$, there exists $\bar r$ such that for any $r \ge \bar r$, we have
    
    \begin{equation*}
        f(x^{r+1})- f(x^r)\leq -\eta^r \xi \Vert\Enabla{D}{x^r}{l} \Vert^2 
    \end{equation*}
    
    Since $\lim_{r\in \mathcal{K}\rightarrow\infty} x^r=\bar x$ and $f(x)$ is continuous, we have $\lim_{r\rightarrow\infty}f(x^r)=f(\bar x)$. Taking summation on both sides from $r=\bar r$ to $\infty$, we can obtain that 
    
    \begin{equation*}
        \sum_{r=\bar r}^\infty\eta^r\xi \Vert\Enabla{D}{x^r}{l} \Vert^2\leq f(x^{\bar r}) - \lim_{r\rightarrow\infty} f(x^r)
    \end{equation*}
    
    Since $\sum_{r=\bar r}^\infty \alpha^r = +\infty$, we have $\lim_{r\in\mathcal{K}\rightarrow\infty}\Vert\Enabla{D}{x^r}{l} \Vert^2 = 0$, hence $\mathbb{E}_{D\sim f_D(\bar x, D)} [\nabla_x l(\bar x, D)] = 0$ and the proof is completed.

\end{proof}

\begin{proof}{Proof of Proposition \ref{prop::converge-stationary-armijo}}
    To simplify the denotation, we omit the limitation of $N\rightarrow\infty$. Therefore the descent direction is $d^r=-\Enabla{D}{x^r}{l}$, where $x^r = \lim_{N\rightarrow \infty}x^r_N$. And we denote $f(x) = \E{D}{x^r}{l}$. According to the armijo principle: 
    
    % We prove by contradiction. Suppose there exists a converging subsequence $\lim_{r(\in \mathcal{K})\rightarrow\infty}x^r=\bar x$ and $\Vert \Enabla{D}{\bar x}{l}\Vert > \frac{\epsilon L_1}{1-\sigma}$. 
    
    \begin{equation*}
        f(x^r)-f(x^{r+1})\ge-\eta^r\sigma \Vert\Enabla{D}{x^r}{l}\Vert^Td^r
    \end{equation*}

    Since $\lim_{r(\in \mathcal{K})\rightarrow\infty}\sup_r\Vert\Enabla{D}{x^r}{l}\Vert\ge 0$. The sequence $\E{D}{x^r}{l}$ decreases monotonically and have a lower bound. Thus 
    
    \begin{equation*}
        \lim_{r(\in \mathcal{K})\rightarrow\infty} f(x^r) - f(x^{r+1}) = 0,
    \end{equation*}
    
    which is followed by 
    
    \begin{equation*}
        \lim_{r(\in \mathcal{K})\rightarrow\infty}\eta^r=0
    \end{equation*}
    
    Hence, by the definition of the armijo rule, we must have for some index $\bar r\ge 0$

    \begin{equation*}
        f(x^r) - f(x^r + \frac{\eta^r}{\beta}d^r)<-\sigma\frac{\eta^r}{\beta}\Vert\Enabla{D}{x^r}{l}\Vert^Td^r, \forall r\in\mathcal{K}, r\ge \bar r
    \end{equation*}

    We denote

    \begin{equation*}
        p^r = \frac{d^r}{\Vert d^r\Vert},\quad \bar\eta^r=\frac{\eta^r\Vert d^r\Vert}{\beta}
    \end{equation*}
    
    Since $\Vert p^r\Vert = 1$, there exists a subsequence $\{p^r\}_{\bar {\mathcal{K}}}$ of $\{p^r\}_\mathcal{K}$ such that $\{p^r\}_{\bar {\mathcal{K}}}\rightarrow \bar p$, where $\bar p$ is a unit vector. Then

    \begin{equation*}
        \frac{f(x^r) - f(x^{r+1})}{\bar \eta^r}<-\sigma (\Enabla{x^r}{D}{l})^T p^r
    \end{equation*}

    Hence, 

    \begin{equation}
        \label{eq::appd-armijo-proof-mean-value-before}
        \frac{f(x^r) - \mathbb{E}_{D\sim f_D(x^r)}[l(x^{r+1},D)] + \mathbb{E}_{D\sim f_D(x^r)}[l(x^{r+1},D)] - f(x^{r+1})}{\bar \eta^r}<-\sigma (\Enabla{x^r}{D}{l})^T p^r
    \end{equation}

    By Lemma \ref{lem::eps-distribution-distance}, $f(x^r) - f(x^{r+1})\ge -\epsilon L_1\Vert \bar\eta^rp^r \Vert$

    By using the mean value theorem, 

    \begin{equation}
        \label{eq::appd-armijo-proof-mean-value-after}
        \begin{aligned}
            \frac{-\epsilon L_1\Vert \bar\eta^rp^r \Vert}{\bar \eta^r} &+ \mathbb{E}_{D\sim f_D(x^r)}[\nabla l(x^r+\tilde{\alpha}^rp^r)]^Tp^r \\
            &< -\sigma (\Enabla{x^r}{D}{l})^T p^r
        \end{aligned}
    \end{equation}

    Let $r(\in \bar{\mathcal{K}})\rightarrow \infty$,

    \begin{equation*}
        -\epsilon L_1 - (\Enabla{D}{\bar x}{l})^T \bar p < -\sigma \Enabla{D}{\bar x}{l}^T\bar p
    \end{equation*}

    Substituting $d^r=-\Enabla{D}{x^r}{l}$, we have 

    \begin{equation*}
        -\epsilon L_1 < -(1-\sigma) \Vert\Enabla{D}{\bar x}{l}\Vert
    \end{equation*}

    which completes the proof.

\end{proof}

\begin{proof}{Proof of Theorem \ref{theo::nece-cond}}
    We proof the theorem by contradiction. Suppose $x^*$ maximize $\max_{x} f(x) = \E{D}{x}{l}$, and $\Vert \Enabla{D}{x^*}{l} \Vert > L_1\epsilon$. Then for any $x_1\in X$,
    
    \begin{equation*}
    \begin{aligned}
        f(x_1)-f(x^*) =& \left( \E{D}{x_1}{l}-\E{D}{x}{l} \right) \\
        =& \left( \E{D}{x_1}{l} - \mathbb{E}_{D\sim f_D(x^*)}[l(x_1,D)] \right) \\
        &+ \left( \mathbb{E}_{D\sim f_D(x^*)}[l(x_1,D)] - \E{D}{x}{l} \right) \\
    \end{aligned}
    \end{equation*}
    
    From Lemma \ref{lem::eps-distribution-distance}, we have 
    
    \begin{equation*}
        | \E{D}{x_1}{l} - \mathbb{E}_{D\sim f_D(x^*)}[l(x_1,D)] | \leq L_1\epsilon\Vert x_1-x^* \Vert
    \end{equation*}
    
    For the second term, we expand $l(x_1, D)$ at $x^*$ and obtain
    
    \begin{equation*}
        \mathbb{E}_{D\sim f_D(x^*)}[l(x_1,D)] - \E{D}{x}{l} = \mathbb{E}_{D\sim f_D(x^*)}[\nabla_x l(x^*,D)] ^T (x_1-x^*) + o(\Vert x_1-x^* \Vert)
    \end{equation*}
    
    where $o(\Vert x_1-x^* \Vert)$ denotes the first-order infinitesimals to $\Vert x_1-x^* \Vert$. By substituting the two terms above and divide both sides by $\Vert x_1-x^* \Vert$, we obtain
    
    \begin{equation*}
        \frac{f(x_1)-f(x^*)}{\Vert x_1-x^* \Vert} \ge -L_1\epsilon + \mathbb{E}_{D\sim f_D(x^*)}[\nabla_x l(x^*,D)] ^T \frac{(x_1-x^*)}{\Vert x_1-x^* \Vert} + \frac{o(\Vert x_1-x^* \Vert)}{\Vert x_1-x^* \Vert}
    \end{equation*}
    
    We let $x_1-x$ take the same direction of $\Enabla{D}{x^*}{l}$, hence the second term on the right side becomes $\Vert \Enabla{D}{x^*}{l} \Vert$. Therefore, for any $\xi > 0$, there exists $x_1$ that is sufficiently close to $x^*$ such that 
    
    \begin{equation*}
        \frac{f(x_1)-f(x^*)}{\Vert x_1-x^* \Vert} \ge -L_1\epsilon + \Vert \Enabla{D}{x^*}{l} \Vert - \xi
    \end{equation*}
    
    Since $\Vert \Enabla{D}{x^*}{l} \Vert > L_1\epsilon$ and $\xi$ can be sufficiently small, we have $f(x_1)-f(x^*) > 0$, which contradicts with the condition that $f(x^*)$ is the optimal solution. 
            
\end{proof}

% \begin{proof}{Proof of Corollary \ref{coro::nonsmooth-pricing-convergence}}
%     To simplify the denotation, we denote the objective function $\mathbb{E}_{D\sim f_D(p,z)}[\pi(p,q,D)]$ as $f(p,q)$. We further denote the solution $(p,q)=x$, and $D=D$. And we can perform the same proof steps as the proof of Proposition \ref{prop::converge-stationary-armijo} by substituting $\Enabla{D}{x^r}{l}$ by an element in the subgradient set $g\in \mathbb{E}_{D\sim f_D(x)}[\partial \pi(x,D)]$. 

%     The only problem lies in (\ref{eq::appd-armijo-proof-mean-value-after}). Since $\pi(x,D)$ is not smooth, the we cannot use the mean value theorem directly. We discuss two cases: 
    
%     (a) case 1: when $q^r, q^{r+1}$ of $x^r, x^{r+1}$ in (\ref{eq::appd-armijo-proof-mean-value-before}) lie in the same side of $D$, we can use the mean value theorem directly since $\mathbb{E}_{D\sim f_D(x)}[\partial \pi(x,D)] = \{ \Enabla{D}{x}{\pi} \}$ so (\ref{eq::appd-armijo-proof-mean-value-after}) is true.

%     (b) case 2: when $q^r, q^{r+1}$ of $x^r, x^{r+1}$ in (\ref{eq::appd-armijo-proof-mean-value-before}) lie in different side of $D$. Without the loss of generality, we assume $q^r < D < q^{r+1}$. There exists $\tilde{\eta^r}\in [0, \bar \eta^r]$ such that the second term of $x^* = x^r + \tilde{\alpha^r}p^r$ is $D$. And

%     \begin{equation*}
%         \Enabla{D}{x^r}{\pi} = ((D\wedge q^r), p-c)^T, \quad \Enabla{D}{x^{r+1}}{\pi} = ((D\wedge q^{r+1}), s-c)^T
%     \end{equation*}

%     Since $\partial \pi(x^*,D) = \{(D, p-c - (p-s)e)^T, e\in [0, 1]\}$

%     We have 

% \end{proof}

\begin{proof}{Proof of Theorem \ref{theo::convergence-convex}}
    We analyze the error of $x^{k+1}_N$: 
    
    \begin{equation*}
    \begin{aligned}
        \Vert x^{k+1}_N - x^* \Vert^2 =& \Vert x^k_N - \eta^k G^N(x^k_N,D) - x^*\Vert^2\\
        =& \Vert x^k_N - x^*\Vert^2 - 2\eta^k G^N(x^k_N,D)^T(x^k_N - x^*) + (\eta^k)^2\Vert G^N(x^k_N,D) \Vert^2
    \end{aligned}
    \end{equation*}
    
    let $N\rightarrow +\infty$ for both sides, denote $\lim_{N\rightarrow\infty} x^k_N$ as $x^k$ for simplicity. Since $\lim_{N\rightarrow\infty} G^N(x^k_N,D) = \Enabla{D}{x^k}{l}$ for any $x$, we have 
    
    \begin{equation*}
        \Vert x^{k+1} - x^* \Vert^2 = \Vert x^k - x^*\Vert^2 - 2\eta^k \Enabla{D}{x^k}{l}^T(x^k - x^*) + (\eta^k)^2 \Vert \Enabla{D}{x^k}{l} \Vert^2
    \end{equation*}
    
    We bound the second term by convexity of the cost function
    
    \begin{equation*}
    \begin{aligned}
    \Enabla{D}{x^k}{l}^T(x^k - x^*) =& \mathbb{E}_{D\sim f_D(x^k)}[\nabla l(x^k,D)^T(x^k-x^*)] \\ 
    \ge& \mathbb{E}_{D\sim f_D(x^k)}[ l(x^k,D) - l(x^*,D) ] \\
    \end{aligned}
    \end{equation*}
    
    For the third term, we bound by Assumption \ref{assum::conv-lip-cont}. 
    
    \begin{equation*}
    \Vert \Enabla{D}{x^k}{l} \Vert^2 \leq L^c_3 
    \end{equation*}
    
    Thus 
    
    \begin{equation*}
        2\eta^k \mathbb{E}_{D\sim f_D(x^k)}[l(x^k,D) - l(x^*,D)] \leq -\Vert x^{k+1} - x^* \Vert^2 + \Vert x^k - x^*\Vert^2  + (\eta^k)^2 (L^c_3)^2
    \end{equation*}
    
    We further investigate the right side. Since 
    
    \begin{equation*}
    \begin{aligned}
    \mathbb{E}_{D\sim f_D(x^k)}[l(x^k,D) - l(x^*,D)] =& -\mathbb{E}_{D\sim f_D(x^k)}[l(x^*,D)] + \E{D}{x^*}{l}\\
    &- \E{D}{x^*}{l} + \mathbb{E}_{D\sim f_D(x^k)}[l(x^k,D)]\\
    \ge& -| \mathbb{E}_{D\sim f_D(x^k)}[l(x^*,D)] - \E{D}{x^*}{l} |\\
    &- \E{D}{x^*}{l} + \mathbb{E}_{D\sim f_D(x^k,D)}[l(x^k,D)]\\
    \ge& - \epsilon L^c_2 \Vert x^* - x^k\Vert - \E{D}{x^*}{l} + \mathbb{E}_{D\sim f_D(x^k,D)}[l(x^k,D)]
    \end{aligned}
    \end{equation*}
    
    Thus 
    
    \begin{equation*}
    \begin{aligned}
        2\eta^k (\E{D}{x^k}{l} - \E{D}{x^*}{l}) \leq& 2\eta^k\epsilon L^c_2\Vert x^*-x^k\Vert - \Vert x^{k+1}-x^*\Vert^2\\
        &+ \Vert x^{k}-x^*\Vert^2 + (\eta^kL^c_3)^2
    \end{aligned}
    \end{equation*}
    
    Take summation from $r=0$ to $k$ and take the minimum of the left side, we obtain
    
    \begin{equation*}
    \begin{aligned}
        (2\sum_{r=0}^k\eta^r) \min_{0\leq r\leq k}\{\E{D}{x^r}{l} - \E{D}{x^*}{l}\} \leq& 2\epsilon L^c_2 \sum_{r=0}^k \eta^r\Vert x^*-x^k\Vert\\
        &+ \Vert x^0-x^k\Vert^2 + (L^c_3)^2\sum_{r=0}^k (\eta^r)^2	
    \end{aligned}
    \end{equation*}
    
    Hence complete the proof.
    
\end{proof}

\begin{proof}{Proof of Corollary \ref{coro::price-only-converge}}
	We need to determine $L^c_2$, $L^c_3$ in Theorem \ref{theo::convergence-convex}. Since $\nabla_pl(p,q,D) = -(D\wedge q)$, we have $L^c_3\leq(D\wedge q)$. We can rewrite the cost function to $l(p,q,D) = -(p-c)q + (p-s)(q-D)^+$, therefore, $L^c_2 = 1$. Then we substitute $L^c_2$ and $L^c_3$ into Theorem \ref{theo::convergence-convex} complete the proof.
\end{proof}

\begin{proof}{Proof of Theorem \ref{theo::strconv-dist-stable}}
	We investigate the distance between $x^k_N$ and a stable point $x_{PS}$.
	\begin{equation*}
	\begin{aligned}
		\Vert x^{k+1}_N - x_{PS} \Vert^2 =& \Vert x^k_N -\eta G^N(x^k_N;z) - x_{PS}\Vert^2 \\
		=& \Vert x^k_N-x_{PS}\Vert^2 - 2\eta G^N(x^k_N)^T(x_N^k-x_{PS}) + (\eta^2)\Vert G^N(x^k_N)\Vert^2
	\end{aligned}
	\end{equation*}

    We begin by upper bounding the second term. From Proposition \ref{prop::grad-converge}, we know that for any $\xi>0$, there exists a sample size $N_0$ such that $\sup_{x}\Vert G^N(x) - \Enabla{D}{x}{l}\Vert\leq\xi$ for all $N>N_0$. Thus we have
    
    \begin{equation*}
    \begin{aligned}
        G^N(x^k_N)^T(x^k_N-x_{PS})\ge& \Enabla{D}{x^k_N}{l}^T(x^k_N-x_{PS})-\xi\Vert x^k_N-x_{PS}\Vert
    \end{aligned}
    \end{equation*}
    
    We can further bound the first term using the same approach as the proof of proposition 2.3 in \cite{DAU5-Perdomo2020}'s work. They give that 
    
    \begin{equation*}
        \Enabla{D}{x^k_N}{l}^T(x^k_N-x_{PS}) \ge B\Vert x^k_N-x_{PS} \Vert^2
    \end{equation*}
    
    We then bound the third term: 
    
    \begin{equation*}
        \Vert G^N(x^k_N)\Vert^2 \leq \xi^2 + \Vert \Enabla{D}{x^k_N}{l} \Vert^2
    \end{equation*}
    
    We can also adopt the same approach in the proof of proposition 2.3 in \cite{DAU5-Perdomo2020}. They give that under Assumption \ref{assum::conv-lip-cont} and \ref{assum::eps-sensitive},
    
    \begin{equation*}
        \Vert \Enabla{D}{x^k_N}{l} \Vert^2 \leq 2B^2\Vert x^k_N - x_{PS}\Vert^2
    \end{equation*}
    
    Therefore, we obtain
    
    \begin{equation}\label{app::eq::strconvex1}
        \Vert x^{k+1}_N - x_{PS} \Vert^2 \leq (1-2\eta A+2\eta^2B^2)\Vert x^k_N-x_{PS}\Vert^2 + 2\eta \xi \Vert x^k_N-x_{PS}\Vert + \xi^2\eta^2
    \end{equation}
    
    In case 1, to give a reasonable distance bound, we need to choose $\eta$ such that the right-hand side is a perfect quadratic polynomial. Thus we choose $\eta$ such that
    
    \begin{equation*}
        4B^2\eta^2-2A\eta+1 = 0
    \end{equation*}
    
    Note that from the Viete's theorem, the two solutions are both positive since we assume $A>0$. Thus we only need to ensure that the equation have real solution, that is
    
    \begin{equation*}
        4A^2\ge 16B^2, \quad A\ge 2B
    \end{equation*}
    
    And we take the square root two both sides of equation (\ref{app::eq::strconvex1}) 
    
    \begin{equation*}
        \Vert x^{k+1}_N - x_{PS} \Vert \leq \sqrt{1-2\eta A+2\eta^2B^2}\Vert x^k_N-x_{PS}\Vert + \xi\eta
    \end{equation*}
    
    We denote $C=\sqrt{1-2\eta A+2\eta^2B}$ and divide both sides by $C^{k+1}$
    
    \begin{equation*}
        \frac{\Vert x^{k+1}_N - x_{PS} \Vert}{C^{k+1}} \leq \frac{\Vert x^k_N-x_{PS}\Vert}{C^k} + \frac{\xi\eta}{C^{k+1}}
    \end{equation*}
    
    Take the summation on both sides from $0$ to $k+1$ and we obtain
    
    \begin{equation*}
        \Vert x^{k+1}_N-x_{PS}\Vert\leq C^{k+1}\Vert x^1_N-x^*\Vert + \xi \eta \frac{1-C^{k+1}}{1-C}
    \end{equation*}
    
    Note that $\eta A-\eta^2B^2 = \frac{2\eta A+1}{4}>0$, thus $C<1$ and the distance is decreasing. 
    
    Now we focus on case 2. Since the quadratic term on the right-hand side of (\ref{app::eq::strconvex1}) is less than zero, we obtain
    
    \begin{equation}
        \Vert x^{k+1}_N - x_{PS} \Vert^2 \leq 2\eta \xi \Vert x^k_N-x_{PS}\Vert + \xi^2\eta^2
    \end{equation}
    
    Thus 
    
    \begin{equation*}
        \Vert x^{k+1}_N - x_{PS} \Vert^2 - \Vert x^{k}_N - x_{PS} \Vert^2  \leq -\Vert x^{k}_N - x_{PS} \Vert^2 + 2\eta \xi \Vert x^k_N-x_{PS}\Vert + \xi^2\eta^2
    \end{equation*}
    
    If $\Vert x^{k}_N - x_{PS} \Vert \ge (1+\sqrt 2)\xi\eta$, we can derive that $\Vert x^{k+1}_N - x_{PS} \Vert^2 - \Vert x^{k}_N - x_{PS} \Vert^2 \leq 0$, which indicates that although the distance may exceed the bound $(1+\sqrt 2)\xi\eta$ some time, it will decrease immediately until it reach the bound, hence complete the proof.
    
\end{proof}

\begin{proof}{Proof of Corollary \ref{coro::dist-opt}}
	The corollary is proved by imposing the triangular inequality to Lemma \ref{lem::stable-opt-distance} and Theorem \ref{theo::strconv-dist-stable}.
\end{proof}

\begin{proof}{Proof of Corollary \ref{coro::price-adj-converge}}
	We observe that $l_a$ is $\gamma$-strongly convex on $p$. We first give the value of $L^c_1$ and $L_D$ in the price-adjustment case. Since $l_a(p,q,D) = (p-\tilde{p})^2 - (p-c)q + (p-s)(q-D)^+$, we have $L_D = 1$. And $\nabla_p l_a(p,q,D) = -(q\wedge D) + 2\gamma(p-\tilde{p})$. we have $L_1^c\leq 2\gamma |p_m-\tilde{p}|$. Then the proof is complete by using the conclusion in Corollary \ref{coro::dist-opt}.
\end{proof}

\section{Experiment Supplements}

\subsection{Description of data}
\label{sec::appd-data-description}

% 真实数据, 仿真数据, 参数设定,

% Table generated by Excel2LaTeX from sheet '参数解释'
\begin{table}[htbp]
    \centering
    \caption{Description of real electricity pricing data}
      \begin{tabular}{ccp{14.055em}p{12.555em}}
      \toprule
      Variable & Type  & Description & Statistics \\
      \midrule
      Date  & datetime & the date of the recording & min 1Jan15, max 6Oct20 \\
      Demand & float & a total daily electricity demand in MWh & min 85.1k, median 120k, max 171k \\
      RRP   & float & a recommended retail price in AUD\$/MWh & min 0, median 66.7, max 300 \\
      min temperature & float & minimum temperature during the day in Celsius & min 0.6, median 11.3, max 28 \\
      max temperature & float & maximum temperature during the day in Celsius & min 9, median 19.1, max 43.5 \\
      solar exposure & float & total daily sunlight energy in MJ/m\^2 & min 0.7, median 12.7, max 33.3 \\
      rainfall & float & daily rainfall in mm & min 0, median 0, max 54.6 \\
      school day & boolean & if students were at school on that day & True 69\%, False 31\% \\
      holiday & boolean & if the day was a state or national holiday & True 4\%, False 96\% \\
      \bottomrule
      \end{tabular}%
    \label{tab::real-data-description}%
  \end{table}%
  
\textit{Real data} The real dataset comes from a real-world power plant pricing scenario (https://www.kaggle.com/datasets/aramacus/electricity-demand-in-victoria-australia). This dataset describes the electricity demand and price situation in Victoria, Australia from 2015 to 2020. The distribution of demand can be seen in Figure \ref{fig::real-demand-distribution}. The descriptive information of the real data is shown in Table \ref{tab::real-data-description}. The factors that influence daily demand are \textit{price, temperature, solar exposure, school day} and \textit{holiday}. Note that we perform an artificial transformation on the temperature. We define heating degree day (HDD) as  $HDD = (15 - T_{max})^+$, and cooling heating degree day (CDD) as $CDD = (T_{min} - 18)^+$, where $T_{max}$ and $T_{min}$ are the highest and lowest centigrade temperatures in one day. This transformation can better reflect the relationship between temperature and electricity demand. The demand is sensitive to price, but it also depend on other features such as temperature and holiday. In our work, we consider the temperature, solar, rainfall, school\_day and holiday factors. Note that the scales of features are different, so we standardize the feature to $[0, 1]$ when processing the data. We use Euclidean metric to measure the distance between samples.

\begin{figure}[htbp]
    \centering
    \includegraphics[width=0.6\linewidth]{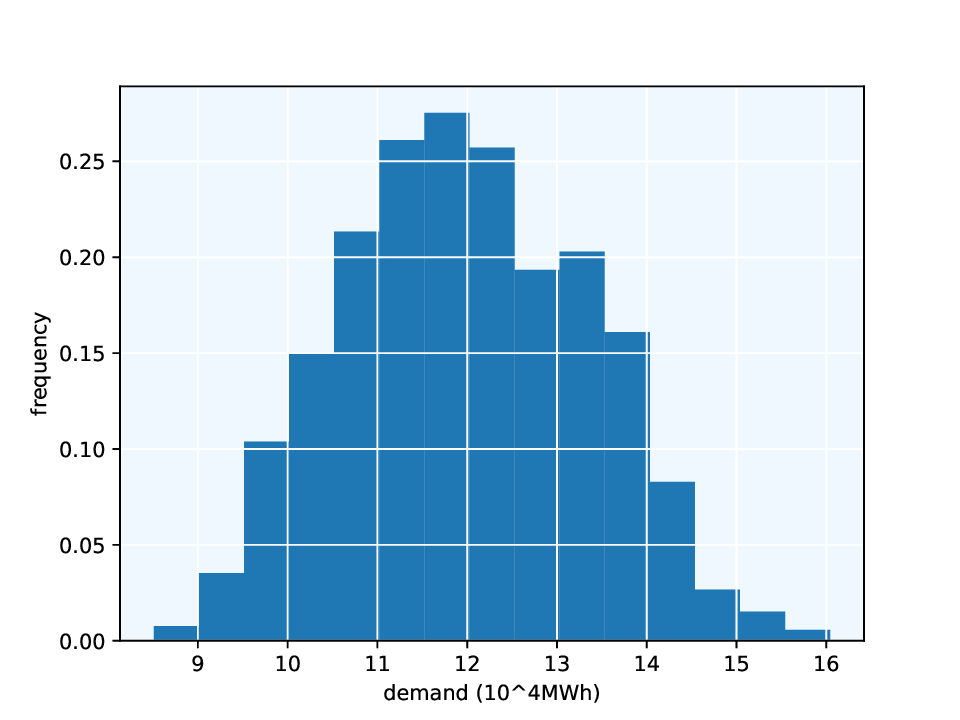}
    \caption{Demand distribution for real data}
    \label{fig::real-demand-distribution}
\end{figure}

\textit{Simulation data} In terms of simulation data, we generate the demand by the following model. 

\begin{equation}
    \label{eq::demand-model}
    D = \max\{0, 60-p+12a^T(X+0.25\phi) + 5b^TX\theta \}
\end{equation}

where $\phi\sim N(0, I_4)$ is a $4$-dimensional vector, $\theta\sim N(0, 1)$ is also a Guassian parameter. Both $\theta$ and $\phi$ are the stochastic factors that cause demand fluctuation. The constant vector $a=(0,8,1,1,1)^T$ and $b=(-1,1,0,0)^T$. Note that we refer the demand model to \cite{Shen-2022-POM}. The demand distribution under $p=20$ is shown in Figure \ref{fig::sim-demand-distribution}. We observe that the distribution is skew and long tail, thus hard to predict by simple models such as linear regression.

\begin{figure}[htbp]
    \centering
    \includegraphics[width=0.6\linewidth]{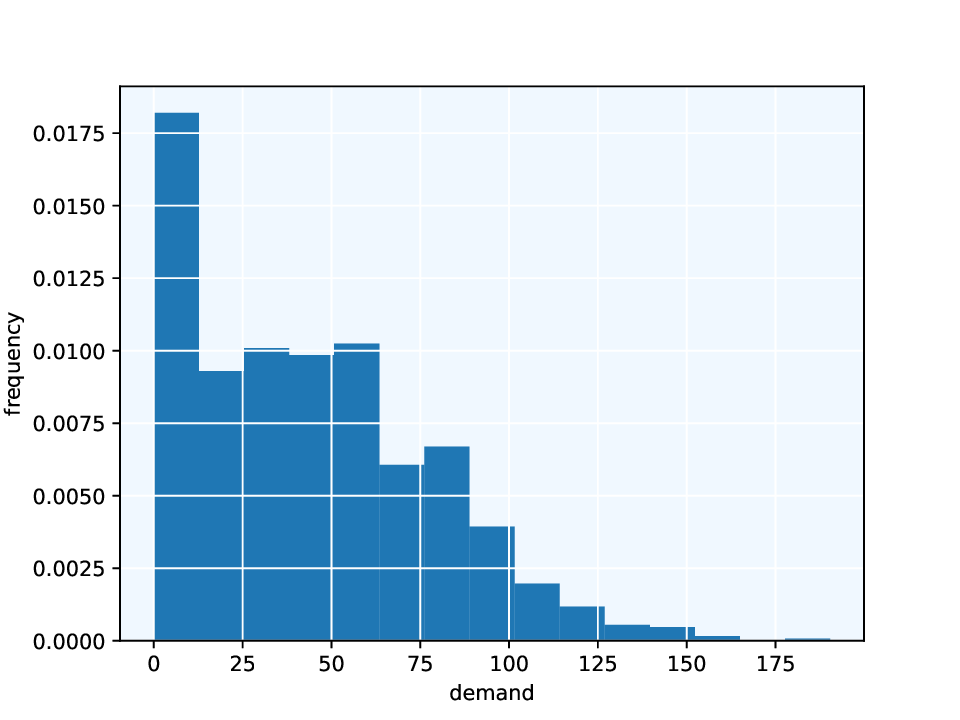}
    \caption{Demand distribution for simulated data}
    \label{fig::sim-demand-distribution}
\end{figure}

The second demand model is simply a linear regression model to suit the PTO assumption, where $D = 60 - p + (1,1,1,1)^Tz + \phi$ and $\phi\sim N(0, 1)$. Thus the demand follows a normal distribution under any price and feature.

\subsection{Step size comparison}
\label{sec::appd-step-compare}

The step size of AGD algorithm adopted in the numerical experiment part is the Armijo step size with $\sigma=0$. In Section \ref{sec::convergence-nonconvex}, we have analyzed the difference between the diminishing step size and Armijo step size. In this section, we will evaluate the difference by experiment.

We first compare two kinds of step size in the simulated dataset. We set the step size constant $C=0.05$ in diminishing step size approach. The realized profit and optimality gap are shown in Table \ref{tab::stepsize-compare}. We can find that the diminishing step size performs worse than Armijo step size in this case. We believe the reason is that the assumptions for the convergence under diminishing step size are usually too strong. The value range $\Omega$ in Assumption \ref{assum::same-limit-range} $L_4$ and $L_5$ constant in Assumption \ref{assum::bound} maybe large in practice, causing a bad convergence performance. Moreover, we find that although any convergent subsequence converge to the local maximum according to Proposition \ref{prop::converge-stationary}, the diminishing step size cannot stop at the local maximum automatically, which indicates that the diminishing step size may not lead to any convergence subsequence. Therefore, the diminishing step size need a careful selection on the step size constant $C$ and stop criteria.

% Table generated by Excel2LaTeX from sheet '步长对比'
\begin{table}[htbp]
    \centering
    \caption{Realized profit of Armijo step size and Diminishing step size}
      \begin{tabular}{ccccc}
      \toprule
      Method & kNN   & kernel & CART  & RF \\
      \midrule
      Armijo & 674.37 & 698.20 & 689.94 & 584.07 \\
      Diminishing & 524.05 & 512.858 & 560.0039 & 388.8681 \\
      \bottomrule
      \end{tabular}%
    \label{tab::stepsize-compare}%
\end{table}%

We also evaluate the effect of hyperparameter $\sigma$ on AGD algorithm. Table \ref{tab::armijo-sigma} reports the optimality gap and iteration number for different constant $\sigma$ under kernel regression . Figure \ref{fig::armijo-sigma} plots the supplementary result in terms of $\sigma$ and optimality gap. We observe the performance is stable when $\alpha_0\in (0.01, 0.1)$, and when $\sigma\leq 0.2$. The increment of both $\alpha_0$ and $\sigma$ can reduce the iteration numbers, thus accelerate the solution. But when $\alpha_0$ is larger than $0.5$, the optimality gap may become larger. Larger $\sigma$ can block the update of solution and may cause the algorithm to stop before reaching the convergence.  

% Table generated by Excel2LaTeX from sheet '步长对比'
\begin{table}[htbp]
    \centering
    \caption{Performance comparison among different initial step sizes and $\sigma$ of Armijo step size}
      \begin{tabular}{cccc|rrrr}
      \toprule
      ($\alpha_0$, $\sigma$) & profit & optimality gap & iterations & \multicolumn{1}{c}{($\alpha_0$, $\sigma$)} & \multicolumn{1}{c}{profit} & \multicolumn{1}{c}{optimality gap} & \multicolumn{1}{c}{iterations} \\
      \midrule
      (0.01, 0) & 695.89 & 0.97\% & 90    & \multicolumn{1}{c}{(0.5, 0)} & 672.02 & 4.37\% & 2 \\
      (0.01, 0.1) & 688.28 & 2.05\% & 74    & \multicolumn{1}{c}{(0.5, 0.1)} & 691.25 & 1.63\% & 2 \\
      (0.01, 0.2) & 659.48 & 6.15\% & 62    & \multicolumn{1}{c}{(0.5, 0.2)} & 667.36 & 5.03\% & 2 \\
      (0.01, 0.5) & 475.37 & 32.35\% & 28    & \multicolumn{1}{c}{(0.5, 0.5)} & 562.69 & 19.92\% & 1 \\
      (0.01, 0.9) & 300.00   & 57.31\% & 0     & \multicolumn{1}{c}{(0.5, 0.9)} & 300.00   & 57.31\% & 0 \\
      (0.05, 0) & 698.20 & 0.64\% & 18    & \multicolumn{1}{c}{(1, 0)} & 653.18 & 7.05\% & 1 \\
      (0.05, 0.1) & 688.84 & 1.97\% & 15    & \multicolumn{1}{c}{(1, 0.1)} & 653.18 & 7.05\% & 1 \\
      (0.05, 0.2) & 662.67 & 5.70\% & 13    & \multicolumn{1}{c}{(1, 0.2)} & 653.18 & 7.05\% & 1 \\
      (0.05, 0.5) & 478.73 & 31.87\% & 6     & \multicolumn{1}{c}{(1, 0.5)} & 569.02 & 19.02\% & 1 \\
      (0.05, 0.9) & 300.00  & 57.31\% & 0     & \multicolumn{1}{c}{(1, 0.9)} & 300.00   & 57.31\% & 0 \\
      (0.1, 0) & 696.72 & 0.85\% & 9     &       &       &       &  \\
      (0.1, 0.1) & 689.42 & 1.89\% & 8     &       &       &       &  \\
      (0.1, 0.2) & 659.40 & 6.16\% & 7     &       &       &       &  \\
      (0.1, 0.5) & 491.16 & 30.10\% & 3     &       &       &       &  \\
      (0.1, 0.9) & 300.00   & 57.31\% & 0     &       &       &       &  \\
      \bottomrule
      \end{tabular}%
    \label{tab::armijo-sigma}%
  \end{table}%

\begin{figure}[htbp]
    \centering
    \includegraphics[width=0.6\linewidth]{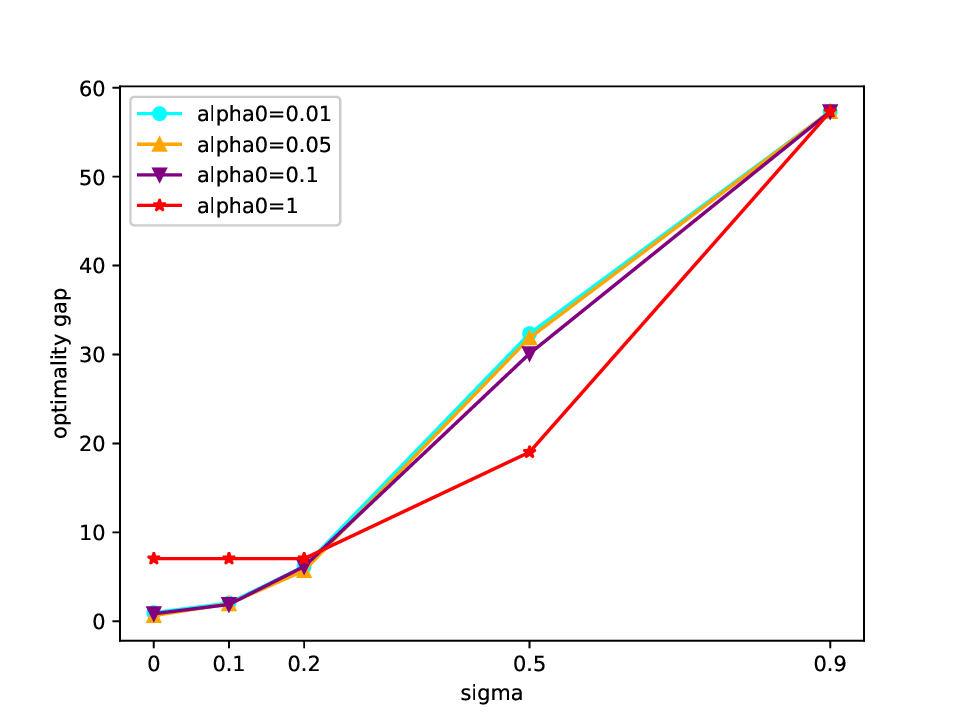}
    \caption{Performance comparison among different initial step sizes and $\sigma$ of Armijo step size}
    \label{fig::armijo-sigma}
\end{figure}

% step size 比较(不同方法，不同步长)